\documentclass[11pt]{amsart}

\usepackage[dvipsnames, svgnames, x11names]{xcolor}
\usepackage{amsmath}
\usepackage{amsfonts}
\usepackage{amssymb}
\usepackage{mathrsfs}
\usepackage[colorlinks, linkcolor = SteelBlue, citecolor = SteelBlue]{hyperref}
\usepackage{enumerate}
\usepackage{enumitem}
\usepackage{extarrows}
\usepackage{etoolbox}
\usepackage{color,xcolor}
\definecolor{cobalt}{RGB}{61,99,181}

\newcommand{\ii}{\mathrm{i}}
\newcommand{\ee}{\mathrm{e}}
\newcommand{\dd}{\,\mathrm{d}}
\newcommand{\rd}{\mathrm{d}}

\newcommand{\CC}{\mathbb{C}}
\newcommand{\DD}{\mathbb{D}}

\newcommand{\TT}{\mathbb{T}}

\newcommand{\Hol}{\mathrm{Hol}}

\DeclareMathOperator{\re}{Re}
\DeclareMathOperator{\rank}{rank}

\DeclareMathOperator{\tr}{tr}

\DeclareMathOperator{\spa}{span}

\newtheorem{thm}{Theorem}[section]

\newtheorem{lem}[thm]{Lemma}
\newtheorem{coro}[thm]{Corollary}
\newtheorem{prop}[thm]{Proposition}
\theoremstyle{definition}

\newtheorem*{ques}{Question}
\theoremstyle{remark}
\newtheorem*{claim}{Claim}
\newtheorem{remark}[thm]{Remark}

\hypersetup{colorlinks = true}

\title[Schatten class composition operators]{Schatten class membership of products of composition operators and their adjoints}
\author[Leng]{Qianrui Leng}
\address{College of Mathematics and Statistics, Chongqing University, Chongqing, 401331, P. R. China}
\email{17788619117@163.com}

\author[Zhao]{Xianfeng Zhao}
\address{College of Mathematics and Statistics, Chongqing University, Chongqing, 401331, P. R. China}
\email{xianfengzhao@cqu.edu.cn}


\date{\today}

\begin{document}	
	\begin{abstract}
		In this paper, we investigate the Schatten class membership of the product of a composition operator and the adjoint of another composition operator on the Hardy and Bergman spaces. We provide several sufficient conditions and necessary conditions for these operators to belong to the Schatten classes. Moreover,  for some specific symbols  including linear fractional maps, finite Blaschke products, and inner functions, we obtain necessary and sufficient conditions.
	\end{abstract}

	\maketitle

	\tableofcontents
	
	\section{Introduction}\label{S1}
	
	Let $\DD$ denote the open unit disk $\{ z \in \CC : |z| < 1 \}$ and $\TT$ be its boundary in the complex plane $\mathbb C$. We write $\rd A$ and $\rd m$ for the normalized Lebesgue area measure on $\DD$ and the normalized Lebesgue arc measure on $\TT$, respectively. Let $L^ p (\TT) = L^ p (\TT, \rd m)$ and $L^ p (\DD) = L^ p (\DD, \rd A)$ be the standard Lebesgue spaces.
	For an analytic self-mapping $\varphi$ of $\DD$ (i.e., $\varphi: \DD \rightarrow \DD$ is analytic), the \textit{composition operator with symbol} $\varphi$ is defined by
	\begin{align*}
		C_{\varphi} f = f \circ \varphi, \qquad f \in \Hol (\DD),
	\end{align*}
where $\Hol (\DD)$ denotes the collection of analytic functions on $\DD$.

	The Hardy  and Bergman spaces  of the unit disk $\mathbb D$ are defined by
	\begin{align*}
		H^ 2 (\DD) = \left\{ f \in \Hol (\DD) : \sup_{0 \leqslant r < 1} \int_{\TT} |f(r \ee^ {\ii t})|^ 2 \dd m(\ee^ {\ii t}) < +\infty \right\}
	\end{align*}
	and  $L_a^ 2 = L^ 2 (\DD) \cap \Hol (\DD)$,
respectively. In this paper, we will consider composition operators on $H^2 (\DD)$ and $L_a^2$. In order to distinguish between these two different composition operators, we refer to them as the ``Hardy--composition operator'' and the ``Bergman--composition operator''. It is well-known that both types of composition operators are bounded, as a consequence of Littlewood's subordinate theorem (see \cite[Theorem 11.1]{Zhu} if necessary). The characterization of compact composition operators was obtained by Shapiro \cite{Sha}. In fact, using the Nevanlinna counting function (see Section \ref{a change of variables formula for integrals} for the definition) he also discussed the essential norms of Hardy--composition operators and Bergman--composition operators:
	\begin{thm}
		Let $\varphi$ be an analytic self-mapping of $\DD$.
		\begin{enumerate}[label = \textup{(\alph*)}]
			\item The essential norm of the Hardy--composition operator $C_{\varphi}$ is given by
			\begin{align*}
				\| C_{\varphi} \|_{\mathrm{e}}^ 2 = \varlimsup_{|z| \rightarrow 1^ -} {\frac{N_{\varphi} (z)}{\log{(1 / |z|)}}},
			\end{align*}
			where $N_{\varphi}$ is the Nevanlinna counting function of $\varphi$.
			\item The essential norm of the Bergman--composition operator $C_{\varphi}$ satisfies that
			\begin{align*}
				m \varlimsup_{|z| \rightarrow 1^ -} {\frac{N_{\varphi, 2} (z)}{[\log{(1 / |z|)}]^ 2}} \leqslant  \| C_{\varphi} \|_{\mathrm{e}}^ 2 \leqslant  \varlimsup_{|z| \rightarrow 1^ -} {\frac{N_{\varphi, 2} (z)}{[\log{(1 / |z|)}]^ 2}},
			\end{align*}
			where $m > 0$ is a constant and $N_{\varphi, 2}$ is the generalized counting function of $\varphi$ defined in \textup{\cite{Sha}}.
		\end{enumerate}
	\end{thm}
	For further discussions on compact composition operators, we refer the reader to \cite{MS, Sar, SS} and \cite{CiMa}. The characterization of composition operators in the Schatten classes was settled by Luecking and Zhu \cite{LZ}, see also \cite{Xia, HX} for further developments. The conclusion is as follows:
	\begin{thm} \label{therorem: characterization of Sp for Cphi}
		Let $\varphi$ be an analytic self-mapping of $\DD$  and $0 < p < \infty$.
		\begin{enumerate}[label = \textup{(\alph*)}]
			\item The Hardy--composition operator $C_{\varphi}$ belongs to the Schatten $p$-class $S_p (H^ 2 (\DD))$ if and only if
			\begin{align*}
				\int_{\DD} \left[ \frac{N_{\varphi} (z)}{\log{(1 / |z|)}} \right]^ {\frac{p}{2}} \, \frac{\rd A(z)}{( 1 - |z|^ 2 )^ 2} < +\infty,
			\end{align*}
			where $N_{\varphi}$ is the Nevanlinna counting function of $\varphi$.
			\item The Bergman--composition operator $C_{\varphi}$ belongs to the Schatten $p$-class $S_p (L_a^ 2)$ if and only if
			\begin{align*}
				\int_{\DD} \left[ \frac{N_{\varphi, 2} (z)}{ \log{(1 / |z|)} } \right]^ {\frac{p}{2}} \frac{\rd A(z)}{(1 - |z|^ 2)^ 2} < +\infty,
			\end{align*}
			where $N_{\varphi, 2}$ is the generalized counting function of $\varphi$ defined in \textup{\cite{Sha}}.
		\end{enumerate}
	\end{thm}
	
	Although the compactness and Schatten class membership of a single composition operator on the Hardy/Bergman space have been completely characterized, the results for the Schatten class of operators of the form  $C_{\psi}^* C_{\varphi}$ and $C_{\varphi} C_{\psi}^*$ remain unknown.
The study of the product of a composition operator and the adjoint of another composition operator originated in the literatures \cite{Cli, CZ}.
In these two papers, the authors established  necessary and sufficient conditions, under varied hypothesis on the induced maps, for the operators $C_{\psi}^*C_{\varphi}$ and $C_{\varphi}C_{\psi}^*$  to be compact on the  Hardy space. The same problem in the setting of the Bergman space  was investigated by Clifford and Zheng \cite{CZ2}. They completely characterized the compactness of the operators $C_{\varphi} C_{\psi}^ *$ and, under the additional assumption that $\varphi$ and $\psi$ are univalent, the compactness of $C_{\psi}^ * C_{\varphi}$. For further studies on the operators $C_{\psi}^ * C_{\varphi}$ and $C_{\varphi} C_{\psi}^ *$, including their numerical range, invertibility and finite-rank property,  we refer the reader to \cite{CDW, CDWZ, CLW, CLN}.
	
	Motivated by conclusions concerning the operators $C_{\psi}^* C_{\varphi}$ and $C_{\varphi} C_{\psi}^*$, it is natural to ask:

	\begin{ques} \label{question}
		When are the operators $C_{\psi}^* C_{\varphi}$ and $C_{\varphi} C_{\psi}^*$ in the Schatten classes?
	\end{ques}

	Studying the above question would provide a precise characterization of these operators in terms of their symbols and yield deeper insights into their compactness. On the one hand, the study of the Schatten $p$-class is much more difficult than that of compact operators. The key difference lies in the singular values: they must be $p$-summable rather than just converging to zero. Furthermore, the space of Schatten $p$-class operators is not closed in the operator norm. This means that standard approximation techniques often fail, as the operator-norm limit of a sequence of operators in the Schatten $p$-class may no longer belong to the Schatten $p$-class.  On the other hand, the technique used to characterize the Schatten class operators of the form $C_\psi^* C_\varphi$ and $C_\varphi C_\psi^*$ differs significantly from those used for a single composition operator. This is because such operators are neither positive nor symmetric. Therefore, the approaches of reducing the problem to positive Toeplitz operators and of using Carleson measures are no longer applicable.

In this paper, by using the  characterizations of Schatten class operators and the Aleksandrov--Clark measures, we investigate conditions under which operators of the form $C_\varphi C^*_\psi$ and $C^*_\psi C_{\varphi}$ belong to the Schatten classes on the Hardy and Bergman spaces.
The structure of the paper is as follows. Section \ref{preliminary} contains some necessary preliminaries that will be used in the sequel. Section \ref{composition operators on the Hardy space} is the main part of this paper, where we focus on the case of the Hardy space. We shall first consider composition operators whose symbols are linear fractional maps, and establish complete characterizations of $C_{\psi}^ * C_{\varphi}$ and $C_{\varphi} C_{\psi}^ *$ in the Schatten classes. See Theorem \ref{theorem: characterization of Sp for linear fractional symbols} for the details. Then we turn to  discussing the operators $C_{\psi}^ * C_{\varphi}$ and $C_{\varphi} C_{\psi}^ *$ with general symbols, and provide criteria for those operators to be in
	\begin{itemize}
		\item the Schatten classes (Theorems \ref{proposition: sufficient condition for Cpsi*Cphi in Sp} and \ref{proposition: sufficient condition for Cphi Cpsi* in Sp}),\vspace{2mm}
		\item the Hilbert--Schmidt class (Propositions \ref{proposition: sufficient condition for S2, Hardy space} and \ref{proposition: sufficient condition for S2, Hardy space, II}, and Remark \ref{remark: proposition: sufficient condition for S2, Hardy space}), and\vspace{2mm}
		\item the trace class (Proposition \ref{proposition: sufficient condition for S1, Hardy space}).
	\end{itemize}
	Furthermore, we also obtain  trace formulas for both  $C_{\psi}^ * C_{\varphi}$ and $C_{\varphi} C_{\psi}^ *$ in Propositions \ref{proposition: trace formula, I} and \ref{proposition: trace formula, II}. Specific symbols such as finite Blaschke products and inner functions are considered in Theorems \ref{theorem: characterization of CB*Cpsi in S2} and \ref{theorem: inner function symbols}.

Since many arguments carry over directly to the Bergman space, we briefly state in Section \ref{composition operators on the Bergman space} the corresponding Bergman space versions of the results from Section \ref{composition operators on the Hardy space}.

	\section{Preliminaries} \label{preliminary}
	
	\subsection{The Hardy and Bergman spaces}\mbox{}
	
	In addition to the definitions given in Section \ref{S1}, the Hardy space can also be defined as a subspace of $L^ 2 (\TT)$ in the following way:
	\begin{align*}
		H^ 2 = \left\{ f \in L^ 2 (\TT) : \int_{\TT} f(\ee^ {\ii t}) \ee^ {\ii n t} \dd m(\ee^ {\ii t}) = 0 \ \textrm{ for } \ n = 1, 2, \cdots \right\}.
	\end{align*}
	Indeed, $H^ 2$ is a closed subspace of $L^ 2 (\TT)$ and hence a Hilbert space. Let $\langle \cdot, \cdot \rangle_2$ and $\| \cdot \|_2$ denote the inner product and the norm on $L^ 2 (\TT)$, respectively. There exists a canonical isometric isomorphism between $H^ 2$ and $H^ 2 (\DD)$.
	Let
	\begin{align*}
		P_z (\ee^ {\ii \theta}) = \frac{1 - |z|^ 2}{|\ee^ {\ii \theta} - z|^ 2} \qquad (\ee^ {\ii \theta} \in \TT)
	\end{align*}
	be the \textit{Poisson kernel} at $z \in \DD$.
	The Poisson integral (harmonic extension) of $f \in H^ 2$ is given by
	\begin{align*}
		\mathcal{P} [f] (z) =  \int_{\TT} f(\ee^ {\ii t}) P_z (\ee^ {\ii t}) \dd m(\ee^ {\ii t}), \qquad z \in \DD,
	\end{align*}
	which is analytic on $\DD$. With a slight abuse of notation, we often write $f(z)$ for $\mathcal{P} [f](z)$. Then the Poisson integral establishes an isometric isomorphism from $H^ 2$ onto $H^ 2 (\DD)$, and the inverse of $\mathcal{P}$ is the radial limit, namely,
	\begin{align*}
		\lim_{r \rightarrow 1^ -} {f (r \ee^ {\ii \theta})} = \lim_{r \rightarrow 1^ -} {\mathcal{P} [f] (r \ee^ {\ii \theta})} = f(\ee^ {\ii \theta})  \ \ \ \ (f \in H^ 2)
	\end{align*}
for  $m$-a.e. $\ee^ {\ii \theta} \in \TT$.
	Moreover, for each $f$ in $H^ 2$, the integral $$\int_{\TT} |f (r \ee^ {\ii t})|^ 2 \dd m(\ee^ {\ii t})$$ increases to $\| f \|_2^ 2$ as $r \rightarrow 1^ -$.
	
	Since each evaluation functional $f \longmapsto f(z)$ is  bounded, $H^ 2$ is in fact a reproducing kernel Hilbert space. It is known that the  reproducing kernel for $H^2$ at $z \in \DD$ is given by
	\begin{align*}
		K_z (\ee^ {\ii \theta}) = \frac{1}{1 - \overline{z} \ee^ {\ii \theta}}, \qquad \ee^ {\ii \theta} \in \TT,
	\end{align*}
	and thus
	\begin{align*}
		f(z) = \langle f, K_z \rangle_2  \ \ \ \  (z \in \DD)
	\end{align*}
for all $f\in H^2$.
	Besides, the normalized reproducing kernel at $z \in \DD$ is given by
	\begin{align*}
		k_z (\ee^ {\ii \theta}) = \frac{K_z (\ee^ {\ii \theta})}{\| K_z \|_2} = \frac{\sqrt{1 - |z|^ 2}}{1 - \overline{z} \ee^ {\ii \theta}}, \qquad \ee^ {\ii \theta} \in \TT.
	\end{align*}
	For more information about the Hardy space, one can consult \cite{Dur, M-AR} and \cite{Zhu}.
	
	Like the Hardy space, the Bergman space $L_a^ 2$ is also a reproducing kernel Hilbert space. One can verify that the reproducing kernel for $L_a^2$ at $z \in \DD$ is
	\begin{align*}
		K_z (w) = \frac{1}{\left( 1 - \overline{z} w \right)^ 2}, \qquad w \in \DD,
	\end{align*}
	and  the associated normalized reproducing kernel at $z\in \mathbb D$ is given by
	\begin{align*}
		k_z (w) = \frac{K_z (w)}{\| K_z \|_{L^ 2 (\DD)}} = \frac{1 - |z|^ 2}{\left( 1 - \overline{z} w \right)^ 2}, \qquad w \in \DD.
	\end{align*}
	While $K_z$ and $k_z$ are used for reproducing kernels for both Hardy and Bergman spaces, their specific references will be apparent from the context. Additionally, for more knowledge about the Bergman space, we refer to the book \cite{Zhu}. 	
	
	Besides $L_a^ 2$, the \textit{weighted Bergman space}
	\begin{align*}
		L_a^ 2 (\rd \sigma) = L^ 2 (\DD, \rd \sigma) \cap \Hol (\DD)
	\end{align*}
	will also be considered, where $\rd \sigma (z) = \log{(1 / |z|^ 2)} \dd A(z)$ and $L^ 2 (\DD, \rd \sigma)$ is the corresponding Lebesgue space. $L_a^ 2 (\rd \sigma)$ is a closed subspace of $L^ 2 (\DD, \rd \sigma)$, and  it is a reproducing kernel Hilbert space. In fact, our introduction of $L_a^ 2 (\rd \sigma)$ is motivated by the following theorem.
	
	\begin{thm}[Littlewood--Paley, \cite{Zhu}] \label{theorem: Littlewood-Paley identity}
		For any two functions $f, g \in H^ 2$, we have the identity
		\begin{align*}
			\langle f, g \rangle_2 = f(0) \overline{g(0)} + \int_{\DD} f'(z) \overline{g'(z)} \dd \sigma(z).
		\end{align*}
	\end{thm}

	Using the above theorem, it is not hard to establish a unitary equivalence between $H^ 2$ and the space $\CC \oplus L_a^ 2 (\rd \sigma)$. Specifically, let
	\begin{align} \label{for1}
		\begin{matrix}
			U: & H^ 2 & \longrightarrow & \CC \oplus L_a^ 2 (\rd \sigma) \vspace{2mm}\\
			& f & \longmapsto &
			\begin{bmatrix}
				f(0) \\
				f'
			\end{bmatrix}
			.
		\end{matrix}
	\end{align}
	By Theorem \ref{theorem: Littlewood-Paley identity}, $U$ is an isometry. Since $U$ is also a surjection as one can easily show, we obtain that $U$ is a unitary operator. By means of $U$, any problem on $H^ 2$ can be transferred to a problem on $\CC \oplus L_a^ 2 (\rd \sigma)$ or even $L_a^ 2 (\rd \sigma)$.
	
	Let us end this subsection with the  following theorem about reproducing kernel Hilbert spaces.
	
	\begin{thm}[{\cite[Theorem 2.4]{PR}}] \label{theorem: reproducing kernels}
		Suppose $\mathcal{H}$ is a reproducing kernel Hilbert space on $X$, and $K_x$ is the reproducing kernel at $x \in X$. For any orthonormal basis  $\{ e_j \}_{j \in \Lambda}$ of $\mathcal{H}$, we have
		\begin{align*}
			K_x (y) = \sum_{j \in \Lambda} {\overline{e_j (x)} e_j (y)},
		\end{align*}
		where the series on the right hand side converges pointwise.
	\end{thm}

	 Noting that $\{ (n + 1) z^ n \}_{n = 0}^ {\infty}$ forms an orthonormal basis  for $L_a^ 2 (\rd \sigma)$, the function
	\begin{align*}
		K_z (w) = \sum_{n = 0}^ {\infty} {(n + 1)^ 2 \left( \overline{z} w \right)^ n} = \frac{1 + \overline{z} w}{\left( 1 - \overline{z} w \right)^ 3}
	\end{align*}
	is the reproducing kernel for $L_a^ 2 (\rd \sigma)$ at $z$.
	
	\subsection{A change of variables formula for integrals} \label{a change of variables formula for integrals} \mbox{}

	For an analytic self-mapping $\varphi$ of $\DD$, its Nevanlinna counting function is defined by
	\begin{align*}
		N_{\varphi} (z) = \sum_{\varphi (w) = z} {\log{\frac{1}{|w|}}}, \qquad z \in \DD \backslash \{ \varphi (0) \},
	\end{align*}
	where the summation accounts for the multiplicity of each root. Littlewood's inequality is an elementary estimate of $N_{\varphi}$:
	\begin{align}
		N_{\varphi} (z) \leqslant \log{\left| \frac{1 - \overline{\varphi (0)} z}{\varphi (0) - z} \right|}, \qquad z \in \DD \backslash \{ \varphi (0) \}. \label{ineq3}
	\end{align}
	A proof of the above inequality can be found in \cite[Corollary 11.3]{Zhu}. Additionally,  a generalization of (\ref{ineq3}) to the case of generalized counting functions can be found in \cite[Proposition 6.3]{Sha}. With the help of the Nevanlinna counting function, a change of variables formula can be established as follows.
	
	\begin{thm}[{\cite[Proposition 11.5]{Zhu}}] \label{theorem: a change of variables}
		If $f$ is nonnegative on $\DD$ and $\varphi$ is an analytic self-mapping of $\DD$, then
		\begin{align*}
			\int_{\DD} f (\varphi (z)) |\varphi' (z)|^ 2 \log{\frac{1}{|z|}} \dd A(z) = \int_{\DD} f(z) N_{\varphi} (z) \dd A(z).
		\end{align*}
	\end{thm}

	\subsection{The composition operator and the Toeplitz operator}\mbox{}
	
	Let $\varphi$ be an analytic self-mapping of $\DD$. Since $H^ 2$ is isometrically isomorphic to $H^ 2 (\DD)$, the composition operator $C_{\varphi}$ on $H^ 2 (\DD)$ also naturally induces an operator on $H^ 2$, which will be denoted by the same notation. As shown in \cite[Proposition 2.25]{CM},
	\begin{align*}
		C_{\varphi} f (\ee^ {\ii \theta}) = f (\varphi (\ee^ {\ii \theta})), \qquad \textrm{$m$-a.e. } \ee^ {\ii \theta} \in \TT \textrm{ for each } f \in H^ 2.
	\end{align*}
	For further discussions about composition operators on the Hardy space, we refer the reader to \cite{Sha2, CM, Zhu}. In addition, one can consult \cite{KS, SZ, IDO, Kos} for studies of composition operators acting on  spaces of analytic functions of several variables.
	
	Toeplitz operators constitute another important class of operators on function spaces. Given an essentially bounded function $u$ on $\DD$, the \textit{Toeplitz operator $T_u$ with symbol $u$} on $L_a^ 2$ is defined by
	\begin{align*}
		T_u f = P (u f), \qquad f \in L_a^ 2,
	\end{align*}
	where $P$ denotes the orthogonal projection from $L^ 2 (\DD)$ onto $L_a^ 2$. Using the reproducing kernels, we can express $T_u$ as an integral operator on $L_a^ 2$:
	\begin{align*}
		T_u f(z) = \langle u f, K_z \rangle_{L^ 2 (\DD)} = \int_{\DD} u(w) f(w) \overline{K_z (w)} \dd A(w), \qquad z \in \DD.
	\end{align*}

	There are also Toeplitz operators whose symbols are not functions but measures. For a finite positive  measure $\mu$ on $\DD$, the \textit{Toeplitz operator $T_{\mu}$ with symbol $\mu$} on the Bergman space $L_a^ 2$ is densely defined by
	\begin{align*}
		T_{\mu} f (z) = \int_{\DD} f(w) \overline{K_z (w)} \dd \mu(w), \qquad z \in \DD,
	\end{align*}
	where $f$ is a bounded analytic function on $\DD$ and $K_z$ is the reproducing kernel for $L_a^ 2$. Unlike Toeplitz operators with bounded function symbols, those with measure symbols are not always bounded. Moreover, Toeplitz operators on $L_a^ 2 (\rd \sigma)$ with bounded function symbols or measure symbols can be  defined in the same way. One can consult the book \cite{Zhu} for detailed discussions of composition operators and Toeplitz operators on the Bergman space $L_a^2$.
	
	We now describe the relationship between these two classes of operators. The argument below follows \cite{Zhu} but dispenses with the assumption that $\varphi (0) = 0$. Note that the extra condition $\varphi (0) = 0$ implies the boundedness of $N_{\varphi} (z) / \log{(1 / |z|)}$ (by Littlewood's inequality  (\ref{ineq3})), whereas our general case does not yield such a bound. Thus, the equality given in (\ref{eq5}) (see below)  requires a more detailed justification.
	
	Let $\varphi$ be an analytic self-mapping of $\DD$. For the unitary operator $U$ defined by (\ref{for1}), it is easy to calculate that
	\begin{align*}
		\bigg( U^ *
		\begin{bmatrix}
			c \\
			g
		\end{bmatrix}
		\bigg) (z)
		= \int_0^ z g (w) \dd w + c, \qquad z \in \DD,
	\end{align*}
where $c\in \mathbb C$.
	It follows that
	\begin{align} \label{eq3}
		U C_{\varphi} U^ *
		\begin{bmatrix}
			c \\
			g
		\end{bmatrix} =
		\begin{bmatrix}
			\int_0^ {\varphi (0)} g (w) \dd w + c \\
			\varphi' \cdot (g \circ \varphi)
		\end{bmatrix}
		.
	\end{align}
	We define an operator $\Lambda_{\varphi}$ by
	\begin{align*}
		\Lambda_{\varphi} g = \int_0^ {\varphi (0)} g (w) \dd w, \ \ \   g\in L_a^ 2 (\rd \sigma),
	\end{align*}
	and let $D_{\varphi}: L_a^ 2 (\rd \sigma) \longrightarrow L_a^ 2 (\rd \sigma)$ be given by
	\begin{align*}
		D_{\varphi} g = \varphi' \cdot (g \circ \varphi).
	\end{align*}
	Then we obtain by (\ref{eq3}) that
	\begin{align*}
		U C_{\varphi} U^ * =
		\begin{bmatrix}
			1 & \Lambda_{\varphi} \\
			0 & D_{\varphi}
		\end{bmatrix}
		.
	\end{align*}
	
 For any two functions $f, g \in L_a^ 2 (\rd \sigma)$,
	\begin{align*}
		\langle D_{\varphi}^ * D_{\varphi} f, g \rangle_{L^ 2 (\DD, \rd \sigma)} & = \langle D_{\varphi} f, D_{\varphi} g \rangle_{L^ 2 (\DD, \rd \sigma)} \\
		& = \int_{\DD} f(\varphi(z)) \overline{g(\varphi (z))} |\varphi' (z)|^ 2 \dd \sigma (z) \\
		& = 2 \int_{\DD} f(z) \overline{g(z)} N_{\varphi} (z) \dd A(z),
	\end{align*}
	where the last ``$=$'' is due to Theorem \ref{theorem: a change of variables}. On the other hand, if we denote
	\begin{align*}
		\Phi (z) = \frac{N_{\varphi} (z)}{\log{(1 / |z|)}},
	\end{align*}
	then
	\begin{align*}
		\langle T_{\Phi} f, g \rangle_{L^ 2 (\DD, \rd \sigma)} = \int_{\DD} \int_{\DD} \Phi (w) f(w) \frac{1 + z \overline{w}}{(1 - z \overline{w})^ 3} \dd \sigma (w) \overline{g(z)} \dd \sigma (z)
	\end{align*}
	for any two bounded analytic functions $f, g$ on $\DD$. 	Noting that if we can use Fubini's theorem to interchange the order of integration in the above expression, then
	\begin{align}
		\langle T_{\Phi} f, g \rangle_{L^ 2 (\DD, \rd \sigma)} & = \int_{\DD} \Phi (w) f (w) \int_{\DD} \overline{g(z)} \frac{1 + z \overline{w}}{(1 - z \overline{w})^ 3} \dd \sigma (z) \rd \sigma (w) \label{eq5} \\
		& = \int_{\DD} f(w) \overline{g(w)} \Phi (w) \dd \sigma (w) \notag \\
		& = \langle D_{\varphi}^ * D_{\varphi} f, g \rangle_{L^ 2 (\DD, \rd \sigma)}, \notag
	\end{align}
	which yields  that $D_{\varphi}^* D_{\varphi} = T_{\Phi}$ on the space of bounded analytic functions. Consequently, $T_{\Phi}$ is bounded on $L_a^ 2 (\rd \sigma)$ and equals  $D_{\varphi}^* D_{\varphi}$.

To verify the conditions of Fubini's theorem, first we observe that
	\begin{align}
		\int_{\DD} \int_{\DD} |\Phi (w) f(w) g(z)| \left| \frac{1 + z \overline{w}}{(1 - z \overline{w})^ 3} \right| \dd \sigma (w) \rd \sigma (z) \leqslant C \int_{\DD} \int_{\DD} \frac{N_{\varphi} (w)}{|1 - z \overline{w}|^ 3} \dd A(w) \rd \sigma (z) \label{ineq4}
	\end{align}
	holds for some constant $C > 0$, since $f$ and $g$ are bounded. For each fixed $w \in \DD$, using the identity
	\begin{align*}
		\frac{1}{(1 - z \overline{w})^ {\frac{3}{2}}} = \sum_{j = 0}^ {\infty} {\frac{\Gamma (j + \frac{3}{2})}{j! \Gamma (\frac{3}{2})} (z \overline{w})^ j},
	\end{align*}
	we have
	\begin{align}
		\int_{\DD} \frac{\rd \sigma (z)}{|1 - z \overline{w}|^ 3} & = \sum_{j, k = 0}^ {\infty} {\frac{\Gamma (j + \frac{3}{2}) \Gamma (k + \frac{3}{2})}{j! k! [\Gamma (\frac{3}{2})]^ 2} \overline{w}^j w^k \int_{\DD} z^ j \overline{z}^k \dd \sigma (z)} \notag \\
		& = \sum_{j = 0}^ {\infty} {\frac{[\Gamma (j + \frac{3}{2})]^ 2}{(j!)^ 2 [\Gamma (\frac{3}{2})]^ 2} |w|^ {2 j} \int_{\DD} |z|^ {2 j} \dd \sigma (z)} \notag \\
		& \leqslant C' \log{\frac{1}{1 - |w|^ 2}}, \label{ineq10}
	\end{align}
	where $C' > 0$ is a constant. It follows that
	\begin{align*}
		\int_{\DD} \int_{\DD} |\Phi (w) f(w) g(z)| \left| \frac{1 + z \overline{w}}{(1 - z \overline{w})^ 3} \right| \dd \sigma (w) \rd \sigma (z) & \leqslant C C' \int_{\DD} N_{\varphi} (w) \log{\frac{1}{1 - |w|^ 2}} \dd A(w).
	\end{align*}
	By (\ref{ineq3}), the function $N_{\varphi} (w) \log\big({\frac{1}{1 - |w|^ 2}}\big)$ is integrable over $\DD$ with respect to $\rd A$, which implies that the integral on the left hand side of (\ref{ineq4}) is finite. We summarize the preceding discussions in the following lemma.
	
	\begin{lem} \label{lemma: Cphi* Cphi is a Toeplitz opeartor}
		Suppose $\varphi$ is an analytic self-mapping of $\DD$. Let
		\begin{align*}
			\Phi (z) = \frac{N_{\varphi} (z)}{\log{(1 / |z|)}},
		\end{align*}
		and let $D_{\varphi}$ be the bounded operator on $L_a^ 2 (\rd \sigma)$ defined by $D_{\varphi} f = \varphi' \cdot (f \circ \varphi)$.
		Then the Toeplitz operator $T_{\Phi}$ on $L_a^ 2 (\rd \sigma)$ is bounded and equals $D_{\varphi}^ * D_{\varphi}$. Moreover, we have
		\begin{align*}
			\langle T_{\Phi} f, g \rangle_{L^ 2 (\DD, \rd \sigma)} = \langle D_{\varphi}^ * D_{\varphi} f, g \rangle_{L^ 2 (\DD, \rd \sigma)} = 2  \int_{\DD} f(z) \overline{g(z)} N_{\varphi} (z) \dd A(z),
		\end{align*}
		where $f, g \in L_a^ 2 (\rd \sigma)$.
	\end{lem}

	An analogous result holds  in the setting of the  Bergman  space $L_a^2$.
	
	\begin{lem}[{\cite[Section 11.3]{Zhu}}] \label{lemma: Cphi* Cphi is a Toeplitz opeartor, Bergman space}
		Suppose $\varphi$ is an analytic self-mapping of $\DD$, and $\rd \mu_{\varphi} = \rd A \circ \varphi^ {-1}$ is the pullback measure of $\rd A$ under the mapping $\varphi$. Then we have $C_{\varphi}^ * C_{\varphi} = T_{\mu_{\varphi}}$  and
		\begin{align*}
			\langle T_{\mu_{\varphi}} f, g \rangle_{L^ 2 (\DD)} = \langle C_{\varphi}^ * C_{\varphi} f, g \rangle_{L^ 2 (\DD)} = \int_{\DD} f(z) \overline{g(z)} \dd \mu_{\varphi} (z), \qquad f, g \in L_a^ 2.
		\end{align*}
	\end{lem}

	\subsection{Aleksandrov--Clark measures} \mbox{}
	
	Recall that $P_z (\ee^ {\ii \theta})$ stands for the Poisson kernel at $z \in \DD$. For an analytic self-mapping $\varphi$ of $\DD$, $P_{\varphi (z)} (\ee^ {\ii \theta})$ is a positive harmonic function on $\DD$ for each fixed $\ee^ {\ii \theta} \in \TT$, since $P_z (\ee^ {\ii \theta})$ is positive and harmonic in $z$ on $\DD$. It follows from Herglotz's theorem (see \cite[Theorem 4.12]{Ros}) that
	\begin{align*}
		P_{\varphi (z)} (\alpha) = \int_{\TT} P_z (\ee^ {\ii t}) \dd \mu_{\alpha} (\ee^ {\ii t}), \qquad z \in \DD,
	\end{align*}
	where $\rd \mu_{\alpha}$ is a finite positive Borel measure on $\TT$ depending on $\alpha \in \TT$. Measures in the family $\left\{ \rd \mu_{\alpha} : \alpha \in \TT \right\}$ are called the \textit{Aleksandrov--Clark measures} associated with $\varphi$. Since
	\begin{align*}
		\mu_{\alpha} (\TT) = P_{\varphi (0)} (\alpha) = \frac{1 - |\varphi (0)|^ 2}{|\alpha - \varphi (0)|^ 2} \leqslant \frac{1 + |\varphi (0)|}{1 - |\varphi (0)|},
	\end{align*}
	the family $\{ \rd \mu_{\alpha}: \alpha \in \TT \}$  is bounded with respect to the total variation norm.
	A particularly interesting aspect of Aleksandrov--Clark measures is captured in the following theorem.
	
	\begin{thm}[Aleksandrov's disintegration theorem] \label{theorem: Aleksandrov's disintegration theorem}
		Suppose $\varphi$ is an analytic self-mapping of $\DD$, and $\{ \rd \mu_{\alpha} : \alpha \in \TT \}$ is the family of the Aleksandrov--Clark measures associated with $\varphi$. For each $f \in L^ 1 (\TT)$, we have:
		\begin{enumerate}[label = \textup{(\alph*)}]
			\item $f \in L^ 1 (\TT, \rd \mu_{\alpha})$ for $m$-a.e. $\alpha \in \TT$,
			\item $\int_{\TT} f \dd \mu_{\alpha} \in L^ 1 (\TT, \rd m(\alpha))$, and
			\item the identity
			\begin{align*}
				\int_{\TT} \int_{\TT} f (\ee^ {\ii t}) \dd \mu_{\alpha} (\ee^ {\ii t}) \dd m(\alpha) = \int_{\TT} f (\ee^ {\ii t}) \dd m (\ee^ {\ii t})
			\end{align*}
			holds.
		\end{enumerate}
	\end{thm}

	\begin{proof}
		See \cite[Theorem 6.1]{Ros} and the remarks thereafter.
	\end{proof}
	
	Furthermore, there is a also deep connection between Aleksandrov--Clark measures and composition operators:
	
	\begin{thm}[{\cite[Theorem 5.1]{Sak}}] \label{theorem: Aleksandrov operators}
		Suppose $\varphi$ is an analytic self-mapping of $\DD$. Let $\{ \rd \mu_{\alpha} : \alpha \in \TT \}$ be the family of the Aleksandrov--Clark measures associated with $\varphi$. Then we have
		\begin{align*}
			C_{\varphi}^ * f (\alpha) = \int_{\TT} f (\ee^ {\ii t}) \dd \mu_{\alpha} (\ee^ {\ii t}), \qquad \textrm{$m$-a.e. } \alpha \in \TT,
		\end{align*}
		where $C_{\varphi}$ is the Hardy--composition operator and $f \in H^ 2$.
	\end{thm}

	For a detailed discussion of Aleksandrov--Clark measures, we refer the reader to \cite{Sak, Ros} and the references therein.
	
	\subsection{The Schatten classes} \mbox{}

For the general theory of the Schatten classes, the reader is referred to \cite{Rin}. In this subsection, we only review the essentials necessary for our subsequent analysis. Let $\mathcal{H}$ be a separable Hilbert space of infinite dimension and $\mathfrak{L} (\mathcal{H})$ the space of bounded linear operators on $\mathcal{H}$. For each $1 \leqslant p < \infty$, the Schatten $p$-class $S_p = S_p (\mathcal{H})$ is the set of all $A \in \mathfrak{L}(\mathcal{H})$ satisfying
	\begin{align*}
		\sum_{j = 1}^ {\infty} {|\langle A e_j, e_j \rangle|^ p} < +\infty
	\end{align*}
for any orthonormal basis   $\{ e_j \}_{j = 1}^ {\infty}$ of $\mathcal{H}$. $S_1$ is usually called the \textit{trace class}, and $S_2$ is usually called the \textit{Hilbert--Schmidt class}.

	It can be shown that $S_p$ is a two-sided $*$-ideal of $\mathfrak{L} (\mathcal{H})$, each finite rank operator is in $S_p$, and each $A \in S_p$ is compact. An alternative definition of $S_p$ is: for each $1 \leqslant p < \infty$, $A \in \mathfrak{L} (\mathcal{H})$ is in $S_p$ if and only if $A$ is compact and
	\begin{align}
		\| A \|_{S_p}^ p := \sum_{j \geqslant 1} {\lambda_j^ p} < +\infty, \label{eq1}
	\end{align}
	where $\{ \lambda_j \}_{j \geqslant 1}$ is the sequence of singular values of $A$. Incidentally, (\ref{eq1}) is also used to define the Schatten $p$-class $S_p$ for every $0 < p < 1$. As a consequence of (\ref{eq1}), we immediately find that $A \in S_p$ if and only if $|A| \in S_p$, since the sequences of singular values of $A$ and $|A|$ are the same.
	
	Besides (\ref{eq1}), we can also express $\| \cdot \|_{S_p}$ in terms of the trace, which will be introduced below.  For $A \in S_1$, the \textit{trace} of $A$ is defined by
	\begin{align}
		\tr{(A)} = \sum_{j = 1}^ {\infty} {\langle A e_j, e_j \rangle}, \label{eq2}
	\end{align}
	where $\{ e_j \}_{j = 1}^ {\infty}$ is an orthonormal basis  for $\mathcal{H}$. It can be shown that the quantity on the right hand side of (\ref{eq2}) is independent of the choice of orthonormal bases, and thus $\tr{(A)}$ is well-defined. Observe that
	\begin{align*}
		\| A \|_{S_1} = \tr(|A|).
	\end{align*}
	Furthermore, if $A \in S_p$, then $|A|^ p \in S_1$ and
	\begin{align*}
		\| A \|_{S_p} = \left[ \tr(|A|^ p) \right]^ {\frac{1}{p}},
	\end{align*}
	where $0 < p < \infty$.

	The following theorem is a useful characterization of $S_p$.
	
	\begin{thm}[{\cite[Theorem 2.1.3]{Rin}}] \label{theorem: characterization of Sp}
		Suppose $\mathcal{H}$ is a separable Hilbert space and $A$ is a bounded linear operator on $\mathcal{H}$. Let $1 \leqslant  p < \infty$. Then $A \in S_p$ if and only if there is a sequence $\{ F_n \}_{n = 1}^ {\infty}$ in $\mathfrak{L} (\mathcal{H})$, such that $F_n$ has rank not greater than $n$ and
		\begin{align*}
			\sum_{n = 1}^ {\infty} {\| A - F_n \|^ p} < +\infty.
		\end{align*}
	\end{thm}

	In fact, the condition $\rank{(F_n)} \leqslant n$  in the above theorem can be relaxed to
	\begin{align}
		\rank{(F_n)} \leqslant k n, \qquad n = 1, 2, \cdots, \label{ineq9}
	\end{align}
	where $k$ is some fixed positive integer.  Assuming that (\ref{ineq9}) holds and
	\begin{align*}
		\sum_{n = 1}^ {\infty} {\| A - F_n \|^ p} < +\infty,
	\end{align*}
	then we set
	\begin{align*}
		\widetilde{F}_n =
		\begin{cases}
			0, & 1 \leqslant  n \leqslant k - 1, \vspace{2mm}\\
			F_m, & m k \leqslant n \leqslant (m + 1) k - 1.
		\end{cases}
	\end{align*}
	It is easy to see that $\rank({\widetilde{F}_n})\leqslant n$ and
	\begin{align*}
		\sum_{n = 1}^ {\infty} {\big\| A - \widetilde{F}_n \big\|^ p} = (k - 1)  \| A \|^ p + k \sum_{n = 1}^ {\infty} {\| A - F_n \|^ p}< +\infty,
	\end{align*}
	to obtain that $A$ is in $S_p$.

	When $p = 2$, we have the following convenient characterization of $S_2$.
	
	\begin{thm}[{\cite[Theorem 2.4.3]{Rin}}] \label{theorem: characterization of S2}
		Suppose that $\mathcal{H}$ is a separable Hilbert space. Then $A \in S_2$ if and only if
		\begin{align*}
			\sum_{j = 1}^ {\infty} {\| A e_j \|^ 2} = \sum_{j, k = 1}^ {\infty} {|\langle A e_j, e_k \rangle|^ 2} < +\infty,
		\end{align*}
		where $\{ e_j \}_{j = 1}^ {\infty}$ is some orthonormal basis for $\mathcal{H}$. In this case, the sums above are precisely $\| A \|_{S_2}^ 2$.
	\end{thm}

	It is worth noting that unlike the original definition of $S_2$, the above theorem offers a criterion for $S_2$ that requires checking only one orthonormal basis.
	
	\section{Composition operators on the Hardy space} \label{composition operators on the Hardy space}
	
	\subsection{Linear fractional  symbols} \mbox{}

	Let us begin with composition operators on the Hardy space whose symbols are linear fractional maps. Such a class of operators is more tractable and yields many interesting examples. Recall that Clifford and Zheng \cite{CZ} established a complete characterization of the compactness of $C_{\psi}^ * C_{\varphi}$ and $C_{\varphi} C_{\psi}^ *$, where both $\varphi$ and $\psi$ are linear fractional maps from $\DD$ into itself. Their analysis is based on the following representation theorem for the adjoint of the composition operator whose symbol is a linear fractional map.
	
	\begin{thm}[{\cite[Theorem 9.2]{CM}}] \label{theorem: Cowen's theorem}
		Let $\psi (z) = \frac{a z + b}{c z + d}$ be a linear fractional map from $\DD$ into itself with $a d - b c \neq 0$. Then $\tau (z) = \frac{\overline{a} z - \overline{c} } {-\overline{b} z + \overline{d} }$ maps $\DD$ into itself, $g(z) = \frac{1}{-\overline{b} z + \overline{d}}$ and $h(z) = c z + d$ are bounded on $\DD$, and
		\begin{align*}
			C_{\psi}^ * = T_g C_{\tau} T_{h}^ *.
		\end{align*}
	\end{thm}

	Using the above theorem, Clifford and Zheng proved the following key lemma.
	
	\begin{lem}[{\cite{CZ}}]\label{L-CZ}
		Suppose $\varphi$ and $\psi$ are linear fractional maps from $\DD$ into itself. Let $\tau$ be the map associated with $\psi$ as described in Theorem \ref{theorem: Cowen's theorem}. Then:
		\begin{enumerate}[label = \textup{(\alph*)}]
			\item $C_{\varphi} C_{\psi}^ *$ is compact if and only if $C_{\varphi} C_{\tau}$ is compact, and
			\item $C_{\psi}^ * C_{\varphi}$ is compact if and only if $C_{\tau} C_{\varphi}$ is compact.
		\end{enumerate}
	\end{lem}

	In fact, the proof of Lemma \ref{L-CZ} remains valid for the Schatten classes: one only needs to replace ``compact operators'' by ``Schatten $p$-class operators'' in that  proof.
	
	\begin{lem} \label{lemma: linear fractional symbols}
		Suppose $\varphi$ and $\psi$ are linear fractional maps from $\DD$ into itself. Let $\tau$ be the map associated with $\psi$ as described in Theorem \ref{theorem: Cowen's theorem}. Then for any $0 < p < \infty$, we have:
		\begin{enumerate}[label = \textup{(\alph*)}]
			\item $C_{\varphi} C_{\psi}^ *$ is in $S_p (H^ 2)$ if and only if $C_{\varphi} C_{\tau}$ is in $S_p (H^ 2)$, and
			\item $C_{\psi}^ * C_{\varphi}$ is in $S_p (H^ 2)$ if and only if $C_{\tau} C_{\varphi}$ is in $S_p (H^ 2)$.
		\end{enumerate}
	\end{lem}

	From the characterizations of the Schatten $p$-class for $C_{\varphi} C_{\tau} = C_{\varphi \circ \tau}$ and $C_{\tau} C_{\varphi} = C_{\varphi \circ \tau}$, we obtain:
	
	\begin{thm} \label{theorem: characterization of Sp for linear fractional symbols}
		Suppose $\varphi$ and $\psi$ are linear fractional maps from $\DD$ into itself. Then for any $0 < p < \infty$, we have:
		\begin{enumerate}[label = \textup{(\alph*)}]
			\item $C_{\varphi} C_{\psi}^ *$ is in $S_p (H^ 2)$ if and only if $\varphi (\TT) \cap \psi (\TT) \cap \TT = \varnothing$, and
			\item $C_{\psi}^ * C_{\varphi}$ is in $S_p (H^ 2)$ if and only if $\varphi^ {-1} (\TT) \cap \psi^ {-1} (\TT) \cap \TT = \varnothing$.
		\end{enumerate}
	\end{thm}

	\begin{proof}
		We only prove (a), since the proof of (b) can be done  similarly. We may assume that $\psi$ is not a constant function, for otherwise the conclusion is trivial. Let $\tau$ be the linear fractional map associated with $\psi$ as defined in Theorem \ref{theorem: Cowen's theorem}. Let $\delta (z) = \frac{1}{z}$ and define $\widetilde{f} (z) = \overline{f(\overline{z})}$. Then we have
		\begin{align*}
			\tau = \delta \circ \big( \widetilde{\psi} \big)^ {-1} \circ \delta.
		\end{align*}
		By Lemma \ref{lemma: linear fractional symbols}, it suffices to investigate the Schatten $p$-class membership of $C_{\tau \circ \varphi}$.  In  view of Theorem \ref{therorem: characterization of Sp for Cphi}, this is further reduced to a calculation of $N_{\tau \circ \varphi}$. The Nevanlinna counting function of $\tau \circ \varphi$ is given by
		\begin{align*}
			N_{\tau \circ \varphi} (z) =
			\begin{cases}
				\log{\frac{1}{|(\tau \circ \varphi)^ {-1} (z)|}}, & 0 < |(\tau \circ \varphi)^ {-1} (z)| < 1,\vspace{2mm} \\
				0, & \textrm{otherwise}.
			\end{cases}
		\end{align*}
		Putting
		\begin{align*}
			\Phi = (\tau \circ \varphi)^ {-1} = \varphi^ {-1} \circ \delta \circ \widetilde{\psi} \circ \delta,
		\end{align*}
		we can then rewrite $N_{\tau \circ \varphi}$ as
		\begin{align*}
			N_{\tau \circ \varphi} (z) =
			\begin{cases}
				\log{\frac{1}{|\Phi (z)|}}, & 0 < |\Phi (z)| < 1,\vspace{2mm} \\
				0, & \textrm{otherwise}.
			\end{cases}
		\end{align*}
		Letting $\widehat{\CC}$ denote the extended complex plane, i.e., $\widehat{\CC}=\mathbb C\cup\{\infty\}$, then
		\begin{align*}
			\widetilde{\psi} (\widehat{\CC} \backslash \overline{\DD}) = \widehat{\CC} \backslash\widetilde{\psi} (\overline{\DD}) \supseteq \widehat{\CC} \backslash\overline{\DD},
		\end{align*}
		and hence
		\begin{align*}
			\Phi (\DD) = (\varphi^ {-1} \circ \delta \circ \widetilde{\psi} \circ \delta) (\DD) = (\varphi^ {-1} \circ \delta \circ \widetilde{\psi}) (\widehat{\CC} \backslash \overline{\DD}) \supseteq (\varphi^ {-1} \circ \delta) (\widehat{\CC}\backslash\overline{\DD}) = \varphi^ {-1} (\DD) \supseteq \DD.
		\end{align*}
		Now we  consider the following two cases.
		
		\textbf{Case 1.} $\Phi (\DD) \supseteq \overline{\DD}$. In this case, we have
		\begin{align*}
			\int_{\DD} \left[ \frac{N_{\tau \circ \varphi} (z)}{\log{(1 / |z|)}} \right]^ {\frac{p}{2}} \frac{\rd A (z)}{(1 - |z|^ 2)^ 2} = \int_{\Phi^ {-1} (\DD)} \left[ \frac{N_{\tau \circ \varphi} (z)}{\log{(1 / |z|)}} \right]^ {\frac{p}{2}} \frac{\rd A (z)}{(1 - |z|^ 2)^ 2},
		\end{align*}
		where $p$ is any positive number.
		Since the compact set $\Phi^ {-1} \left( \overline{\DD} \right)$ is contained in $\Phi^ {-1} (\Phi (\DD)) = \DD$, there exists some $\rho \in (0, 1)$ so that $\Phi^ {-1} \left( \overline{\DD} \right) \subseteq \rho \DD$. It follows that
		\begin{align*}
			\int_{\DD} \left[ \frac{N_{\tau \circ \varphi} (z)}{\log{(1 / |z|)}} \right]^ {\frac{p}{2}} \frac{\rd A (z)}{(1 - |z|^ 2)^ 2} & \leqslant \frac{1}{[\log{(1 / \rho)}]^ {\frac{p}{2}}} \frac{1}{(1 - \rho^ 2)^ 2} \int_{\Phi^ {-1} (\DD)} [N_{\tau \circ \varphi} (z)]^ {\frac{p}{2}} \dd A(z).
		\end{align*}
		By Littlewood's inequality (\ref{ineq3}), the integral $$\int_{\DD} \big[N_{\tau \circ \varphi} (z)\big]^{\frac{p}{2}} \dd A(z)$$ is finite. We conclude by Lemma \ref{lemma: linear fractional symbols} and Theorem \ref{therorem: characterization of Sp for Cphi} that $C_{\varphi} C_{\psi}^ *$ is in $S_p (H^ 2)$ for $0 < p < \infty$.
		
		\textbf{Case 2.} $\Phi (\DD) \nsupseteq \overline{\DD}$ (i.e., the boundary of $\Phi (\DD)$ is externally tangent to $\TT$). Then we can find some point $\ee^ {\ii \theta} \in \TT$ such that $\Phi (\ee^ {\ii \theta}) = \omega \in \TT$. By the definition of $\Phi$, we have
		\begin{align*}
			\frac{1}{\varphi (\omega)} = \delta \circ \varphi (\omega) = \widetilde{\psi} \circ \delta (\ee^ {\ii \theta}) = \overline{\psi (\ee^ {\ii \theta})}.
		\end{align*}
		Observing that $|\psi (\ee^ {\ii \theta})\varphi (\omega)| = 1$ and recalling that $\varphi ( \overline{\DD})$ and $\psi ( \overline{\DD} )$ are both  contained in $\overline{\DD}$, we must have $|\psi (\ee^ {\ii \theta})| = |\varphi (\omega)| = 1$. It follows that
		\begin{align*}
			\varphi (\omega) = \psi (\ee^ {\ii \theta}) \in \TT.
		\end{align*}
		By \cite[Theorem 2.1]{CZ}, $C_{\varphi} C_{\psi}^ *$ is not compact, and in particular, it is not in any $S_p (H^ 2)$ for $0 < p < \infty$.
		
		Therefore, for each $0 < p < \infty$, $C_{\varphi} C_{\psi}^ * \notin S_p (H^ 2)$ if and only if $\Phi (\DD) \nsupseteq \overline{\DD}$. In the discussion of Case 2, we have shown that $\Phi (\DD) \nsupseteq \overline{\DD}$ leads to
		\begin{align}
			\varphi (\TT) \cap \psi (\TT) \cap \TT \neq \varnothing, \label{ineq11}
		\end{align}
		which gives the sufficiency part of (a).

		Conversely, if (\ref{ineq11}) holds, then there are points $\xi, \zeta$ and $\omega_0$ on $\TT$ such that
		\begin{align*}
			\varphi (\xi) = \psi (\zeta) = \omega_0,
		\end{align*}
		and a direct calculation then gives $\Phi (\zeta) = \xi$. Consequently, $\Phi (\DD) \nsupseteq \overline{\DD}$ and hence $C_{\varphi} C_{\psi}^ * \notin S_p (H^ 2)$, to obtain the necessity part of (a).

This finishes the proof of Theorem \ref{theorem: characterization of Sp for linear fractional symbols}.
	\end{proof}

	\begin{coro} \label{corollary: characterization of Sp for linear fractional symbols}
		Suppose that $\varphi$ and $\psi$ are linear fractional maps from $\DD$ into itself. Then for any $0 < p < \infty$:
		\begin{enumerate}[label = \textup{(\alph*)}]
			\item $C_{\varphi} C_{\psi}^ *$ is in $S_p (H^ 2)$ if and only if $C_{\varphi} C_{\psi}^ *$ is compact, and
			\item $C_{\psi}^ * C_{\varphi}$ is in $S_p (H^ 2)$ if and only if $C_{\psi}^ * C_{\varphi}$ is compact.
		\end{enumerate}
	\end{coro}

	\begin{proof}
		This is an immediate consequence of Theorem \ref{theorem: characterization of Sp for linear fractional symbols} and  \cite[Theorems 2.1 and 2.3]{CZ}.
	\end{proof}

\begin{remark}
Based on the  characterization of  the compactness of $C_{\psi}^ * C_{\varphi}$ and $C_{\varphi} C_{\psi}^ *$, Clifford and Zheng \cite{CZ} constructed three examples of such operators, including:
	\begin{enumerate}[label = \textup{(\alph*)}]
		\item $C_{\psi}^ * C_{\varphi}$ is not compact, but $C_{\varphi} C_{\psi}^ *$ is compact;
		\item $C_{\varphi} C_{\psi}^ *$ is not compact, but $C_{\psi}^ * C_{\varphi}$ is compact;
		\item $C_{\varphi}$ and $C_{\psi}$ are not compact, but both $C_{\psi}^ * C_{\varphi}$ and $C_{\varphi} C_{\psi}^ *$ are compact.
	\end{enumerate}
	Using Corollary \ref{corollary: characterization of Sp for linear fractional symbols},  we obtain that the operators constructed in (a), (b) and (c) are also examples for the Schatten classes.
\end{remark}

	\subsection{\texorpdfstring{The operator $C_{\psi}^ * C_{\varphi}$}{The Operator Cpsi* Cphi}}\mbox{}

	In \cite{Cli}, Clifford calculated an upper estimate for the essential norm of $C_{\psi}^ * C_{\varphi}$. The key to the calculation lies in estimating the operator  norm $\| Q_n C_{\psi}^ * C_{\varphi} Q_n \|$, where $Q_n$ will be specified later. This, together with Theorem \ref{theorem: characterization of Sp}, also yields  a sufficient condition for $C_{\psi}^ * C_{\varphi} \in S_p (H^ 2)$.
	
	For each positive integer $n$, let $P_n$ be the orthogonal projection from $H^ 2$ onto $\spa{\{ 1, z, \cdots, z^ n \}}$, and $Q_n = I - P_n$, where $I$ stands for the identity operator on $H^2$. Instead of invoking \cite[Lemma 5.2]{Sha} directly, we establish a somewhat more refined result.
	
	\begin{lem} \label{lemma: estimate of Qnf and (Qnf)'}
		Let $Q_n$ be the operator defined above. Then we have
		\begin{align*}
			\left| (Q_n f) (z) \right| \leqslant \frac{|z|^ {n + 1}}{\sqrt{1 - |z|^ 2}} \| f \|_2
		\end{align*}
		and
		\begin{align*}
			\left| (Q_n f)' (z) \right| \leqslant C |z|^ n \left[ \frac{1}{(1 - |z|^ 2)^ {{\frac{3}{2}} }} + \frac{\sqrt{n}}{(1 - |z|^ 2)} + \frac{n}{(1 - |z|^ 2)^ {{\frac{1}{2}} }} \right] \| f \|_{2}
		\end{align*}
		for every positive integer $n$, $f \in H^ 2$, and $z \in \DD$, where $C > 0$ is some constant.
	\end{lem}
	
	\begin{proof}
		Let $a_j$ be the $j$-th Fourier coefficient of $f$. Then
		\begin{align*}
			(Q_n f) (z) = \sum_{j = n + 1}^ {\infty} {a_j z^ j} \qquad \textrm{and} \qquad (Q_n f)' (z) = \sum_{j = n + 1}^ {\infty} {j a_j z^ {j - 1}}.
		\end{align*}
		By the Cauchy--Schwarz inequality we have
		\begin{align*}
			\left| (Q_n f) (z) \right| & \leqslant \left( \sum_{j = n + 1}^ {\infty} {|a_j|^ 2} \right)^ {\frac{1}{2}} \left( \sum_{j = n + 1}^ {\infty} {|z|^ {2 j}} \right)^ {\frac{1}{2}}  \leqslant \frac{|z|^ {n + 1}}{\sqrt{1 - |z|^ 2}} \| f \|_2
		\end{align*}
for every $z$ in $\DD$.
		A similar argument also yields that
		\begin{align*}
			\left| (Q_n f)' (z) \right| & \leqslant |z|^ n \left[ \frac{2}{(1 - |z|^ 2)^ 3} + \frac{2 n - 1}{(1 - |z|^ 2)^ 2} + \frac{n^ 2}{1 - |z|^ 2} \right]^ {{\frac{1}{2}} } \| f \|_{2} \\
			& \leqslant C |z|^ n \left[ \frac{1}{(1 - |z|^ 2)^ {{\frac{3}{2}} }} + \frac{\sqrt{n}}{1 - |z|^ 2} + \frac{n}{(1 - |z|^ 2)^ {{\frac{1}{2}} }} \right] \| f \|_2
		\end{align*}
		for every $z$ in $\DD$,  where  $C > 0$ is a constant.
	\end{proof}

Using the above lemma, we can establish a sufficient condition for the operator $C_{\psi}^ * C_{\varphi}$ to belong to $S_p (H^ 2)$.

	\begin{thm} \label{proposition: sufficient condition for Cpsi*Cphi in Sp}
		Suppose $\varphi$ and $\psi$ are analytic self-mappings of $\DD$, and $1 \leqslant  p < \infty$. Then  $C_{\psi}^ * C_{\varphi} \in S_p (H^ 2)$ provided that
		\begin{align}
			\int_{\DD} \frac{1}{[1 - |\varphi (z)|]^ {\frac{3p}{2}} [1 - |\psi (z)|]^ {\frac{3p}{2}}} \frac{|\varphi' (z)\psi' (z)|}{1 - |\varphi (z)\psi (z)|} \dd \sigma (z) < +\infty, \label{ineq6}
		\end{align}
where $\rd \sigma (z) = \log{(1 / |z|^ 2)} \dd A(z)$.
	\end{thm}

	\begin{proof}
		Assume that (\ref{ineq6}) holds. Let $P_n$ and $Q_n$  be defined as in the second paragraph of this subsection. It is easy to see that
		\begin{align*}
			C_{\psi}^ * C_{\varphi} - \left( C_{\psi}^ * C_{\varphi} P_n + P_n C_{\psi}^ * C_{\varphi} - P_n C_{\psi}^ * C_{\varphi} P_n \right) = Q_n C_{\psi}^ * C_{\varphi} Q_n
		\end{align*}
		and
		\begin{align*}
			\rank{\left( C_{\psi}^ * C_{\varphi} P_n + P_n C_{\psi}^ * C_{\varphi} - P_n C_{\psi}^ * C_{\varphi} P_n \right)} \leqslant 3 (n + 1).
		\end{align*}
		By Theorem \ref{theorem: characterization of Sp}, to show $C_{\psi}^ * C_{\varphi} \in S_p (H^ 2)$ it suffices to show that
		\begin{align*}
			\sum_{n = 1}^ {\infty} {\| Q_n C_{\psi}^ * C_{\varphi} Q_n \|^ p} < +\infty.
		\end{align*}

		Since
		\begin{align*}
			\| Q_n C_{\psi}^ * C_{\varphi} Q_n \| = \sup{\left\{ \left| \left\langle C_{\varphi} Q_n f, C_{\psi} Q_n g \right\rangle_2 \right| : f, g \in H^ 2 \textrm{ with } \| f \|_2\leqslant 1, \| g \|_2 \leqslant 1 \right\}},
		\end{align*}
		we turn our attention to estimating $\left\langle C_{\varphi} Q_n f, C_{\psi} Q_n g \right\rangle_2$, where $f$ and $g$ are functions in $H^ 2$ whose norms are at most $1$. Applying Theorem \ref{theorem: Littlewood-Paley identity}, we get
		\begin{align*}
			\left\langle C_{\varphi} Q_n f, C_{\psi} Q_n g \right\rangle_2 & = \underbrace{(Q_n f) (\varphi (0)) \overline{(Q_n g) (\psi (0))}}_{\alpha_n (f, g)} + \underbrace{\int_{\DD} (Q_n f)' (\varphi (z)) \overline{(Q_n g)' (\psi (z))} \varphi' (z) \overline{\psi' (z)} \dd \sigma (z)}_{\beta_n (f, g)}.
		\end{align*}
		By Lemma \ref{lemma: estimate of Qnf and (Qnf)'}, we have the following two estimates:
		\begin{align*}
			\left| \alpha_n (f, g) \right| \leqslant \frac{|\varphi (0)|^ {n + 1} |\psi (0)|^ {n + 1}}{\sqrt{1 - |\varphi (0)|^ 2} \sqrt{1 - |\psi (0)|^ 2}} =: \alpha_n
		\end{align*}
		and
		\begin{align*}
			\left| \beta_n (f, g) \right| \leqslant  C \sum_{k, l = 0}^ 2 {n^ {\frac{k + l}{2}} \int_{\DD} \frac{|\varphi (z)|^ n |\psi (z)|^ n |\varphi' (z) \psi' (z)|}{\left[ 1 - |\varphi (z)|^2 \right]^ {\frac{3 - k}{2}} \left[ 1 - |\psi (z)|^ 2 \right]^ {\frac{3 - l}{2}}} \dd \sigma (z)} =: \beta_n,
		\end{align*}
		where $C > 0$ is some constant. It follows that
		\begin{align*}
			\| Q_n C_{\psi}^ * C_{\varphi} Q_n \|\leqslant \alpha_n + \beta_n.
		\end{align*}
		The sequence $\{ \alpha_n \}_{n = 1}^ {\infty}$ is $p$-summable, so we only need to show that
		\begin{align}
			\sum_{n = 1}^ {\infty} {\beta_n^ p} < +\infty. \label{ineq7}
		\end{align}
		Note that
		\begin{align*}
			\int_{\DD} |\varphi' (z)|~ |\psi' (z)| \dd \sigma (z) & \leqslant \Big[ \int_{\DD} |\varphi' (z)|^ 2 \dd \sigma (z) \Big]^ {\frac{1}{2}} \Big[ \int_{\DD} |\psi' (z)|^ 2 \dd \sigma (z) \Big]^ {\frac{1}{2}} \\
			& = 2 \Big[ \int_{\DD} N_{\varphi} (z) \dd A(z) \Big]^ {\frac{1}{2}} \Big[ \int_{\DD} N_{\psi} (z) \dd A(z) \Big]^ {\frac{1}{2}},
		\end{align*}
		where the ``$=$" is due to Theorem \ref{theorem: a change of variables}. It follows from (\ref{ineq3}) that the measure $|\varphi' (z)\psi' (z)| \dd \sigma (z)$ is finite. This allows us to apply H\"older's inequality to obtain that
		\begin{align*}
			& \sum_{n = 1}^ {\infty} {n^ {\frac{p (k + l)}{2}} \bigg\{ \int_{\DD} \frac{|\varphi (z)|^ n |\psi (z)|^ n |\varphi' (z)\psi' (z)|}{\left[ 1 - |\varphi (z)|^ 2 \right]^ {\frac{3 - k}{2}} \left[ 1 - |\psi (z)|^ 2 \right]^ {\frac{3 - l}{2}}} \dd \sigma (z) \bigg\}^ p} \\
			& \ \ \ \leqslant C' \sum_{n = 1}^ {\infty} {n^ {\frac{p (k + l)}{2}} \int_{\DD} \bigg\{ \frac{|\varphi (z)|^ n |\psi (z)|^ n}{\left[ 1 - |\varphi (z)|^ 2 \right]^ {\frac{3 - k}{2}} \left[ 1 - |\psi (z)|^ 2 \right]^ {\frac{3 - l}{2}}} \bigg\}^ p |\varphi' (z)\psi' (z)| \dd \sigma (z)} \\
			& \ \ \  \leqslant C'' \int_{\DD} \frac{|\varphi' (z)\psi' (z)|}{\left[ 1 - |\varphi (z)|^ 2 \right]^ {\frac{p (3 - k)}{2}} \left[ 1 - |\psi (z)|^ 2 \right]^ {\frac{p (3 - l)}{2}} \left[ 1 - |\varphi (z)|^ p |\psi (z)|^ p \right]^ {1 + \frac{p (k + l)}{2}}} \dd \sigma (z),
		\end{align*}
		where $k, l \in \{ 0, 1, 2 \}$ and $C', C'' > 0$ are constants. Since
		\begin{align*}
			\left( 1 - |\varphi (z)|^ p |\psi (z)|^ p \right)^ {\frac{p (k + l)}{2}} \geqslant \left( 1 - |\varphi (z)| \right)^ {\frac{p k}{2}} \left( 1 - |\psi (z)| \right)^ {\frac{p l}{2}},
		\end{align*}
		we further obtain that
		\begin{align*}
			&\sum_{n = 1}^ {\infty} {n^ {\frac{p (k + l)}{2}} \bigg\{ \int_{\DD} \frac{|\varphi (z)|^ n |\psi (z)|^ n |\varphi' (z)\psi' (z)|}{\left[ 1 - |\varphi (z)|^ 2 \right]^ {\frac{3 - k}{2}} \left[ 1 - |\psi (z)|^ 2 \right]^ {\frac{3 - l}{2}}} \dd \sigma (z) \bigg\}^ p} \\
			&\ \ \ \leqslant C'' \int_{\DD} \frac{|\varphi' (z)\psi' (z)|}{\left[ 1 - |\varphi (z)|^ 2 \right]^ { \frac{3p}{2}} \left[ 1 - |\psi (z)|^ 2 \right]^ {\frac{3p}{2}} \left[ 1 - |\varphi (z)\psi (z)| \right]} \dd \sigma (z),
		\end{align*}
		which is finite by (\ref{ineq6}). Since $k$ and $l$ in the above estimations are arbitrary numbers in $\{ 0, 1, 2 \}$, we obtain (\ref{ineq7}), and the proof is now complete.
	\end{proof}

	Of all Schatten $p$-classes, the Hilbert--Schmidt class and the trace class are of particular interest. We will concentrate on them for the remainder of this subsection. Let us begin by giving two sufficient conditions under which $C_{\psi}^ * C_{\varphi}$ is a Hilbert--Schmidt operator.
	
	\begin{prop} \label{proposition: sufficient condition for S2, Hardy space}
		Suppose that $\varphi$ and $\psi$ are analytic self-mappings of $\DD$. If
		\begin{align}
			\int_{\TT} \int_{\TT} \frac{\rd m(\ee^ {\ii t}) \rd m (\ee^ {\ii s})}{\big| 1 - \varphi (\ee^ {\ii t}) \overline{\varphi (\ee^ {\ii s})} \big| \big| 1 - \overline{\psi (\ee^ {\ii t})} \psi (\ee^ {\ii s}) \big|} < +\infty, \label{ineq1}
		\end{align}
then  $C_{\psi}^ * C_{\varphi} \in S_2 (H^ 2)$.
	\end{prop}

	\begin{proof}
		According to Theorem \ref{theorem: characterization of S2}, we only need to prove that
		\begin{align*}
			\sum_{j, k = 0}^ {\infty} {|\langle C_{\psi}^ * C_{\varphi} (z^ j), z^ k \rangle_2|^ 2} < +\infty
		\end{align*}
		if (\ref{ineq1}) holds.
		Indeed, for any nonnegative integer $N$ we have
		\begin{align*}
			\sum_{j, k = 0}^ N {|\langle C_{\psi}^ * C_{\varphi} (z^ j), z^ k \rangle_2|^ 2} & = \sum_{j, k = 0}^ N {|\langle C_{\varphi} (z^ j), C_{\psi} (z^ k) \rangle_2|^ 2} \\
			& = \sum_{j, k = 0}^ N {\left| \int_{\TT} \varphi^ j \overline{\psi}^ k \dd m \right|^ 2} \\
			& = \sum_{j, k = 0}^ N {\int_{\TT} \int_{\TT} \varphi (\ee^ {\ii t})^ j \overline{\varphi (\ee^ {\ii s})}^ j \overline{\psi (\ee^ {\ii t})}^ k \psi (\ee^ {\ii s})^ k \dd m(\ee^ {\ii t}) \rd m (\ee^ {\ii s})}.
		\end{align*}
		The assumption in (\ref{ineq1}) implies that both
		\begin{align*}
			\left\{ (\ee^ {\ii t}, \ee^ {\ii s}) \in \TT \times \TT : \varphi (\ee^ {\ii t}) \overline{\varphi (\ee^ {\ii s})} = 1 \right\} \qquad \textrm{and} \qquad \left\{ (\ee^ {\ii t}, \ee^ {\ii s}) \in \TT \times \TT : \overline{\psi (\ee^ {\ii t})} \psi (\ee^ {\ii s}) = 1 \right\}
		\end{align*}
		are null sets on $\TT \times \TT$ with respect to the measure $\rd (m \times m)$. Thus,
		\begin{align*}
			&\sum_{j, k = 0}^ N {|\langle C_{\psi}^ * C_{\varphi} (z^ j), z^ k \rangle_2|^ 2} \\
&\ \ \ = \int_{\TT} \int_{\TT} \frac{1 - \varphi (\ee^ {\ii t})^ {N + 1} \overline{\varphi (\ee^ {\ii s})}^ {N + 1}}{1 - \varphi (\ee^ {\ii t}) \overline{\varphi (\ee^ {\ii s})}} \frac{1 - \overline{\psi (\ee^ {\ii t})}^ {N + 1} \psi (\ee^ {\ii s})^ {N + 1}}{1 - \overline{\psi (\ee^ {\ii t})} \psi (\ee^ {\ii s})} \dd m(\ee^ {\ii t}) \rd m (\ee^ {\ii s}).
		\end{align*}
		The integrand on the right hand side is dominated by the function
		\begin{align*}
			\frac{4}{\big| 1 - \varphi (\ee^ {\ii t}) \overline{\varphi (\ee^ {\ii s})} \big| \big| 1 - \overline{\psi (\ee^ {\ii t})} \psi (\ee^ {\ii s}) \big|},
		\end{align*}
		which is integrable over $\TT \times \TT$ due to  (\ref{ineq1}). Thus,  we conclude that
		\begin{align*}
			\sum_{j, k = 0}^ {\infty} {|\langle C_{\psi}^ * C_{\varphi} (z^ j), z^ k \rangle_2|^ 2} \leqslant 4 \int_{\TT} \int_{\TT} \frac{\rd m (\ee^ {\ii t}) \rd m (\ee^ {\ii s})}{\big| 1 - \varphi (\ee^ {\ii t}) \overline{\varphi (\ee^ {\ii s})} \big| \big| 1 - \overline{\psi (\ee^ {\ii t})} \psi (\ee^ {\ii s}) \big|} < +\infty,
		\end{align*}
		as desired.
	\end{proof}

	\begin{remark}
		Condition (\ref{ineq1}) is not only sufficient but also necessary in the case that $\varphi = \psi$. Suppose that $C_{\varphi}^ * C_{\varphi}$ belongs to $S_2 (H^ 2)$. In this case, the calculation in the proof of the  above proposition  shows that
		\begin{align*}
			\sum_{j, k = 0}^ N {|\langle C_{\varphi}^ * C_{\varphi} (z^ j), z^ k \rangle_2|^ 2} = \int_{\TT} \int_{\TT} \Big| \sum_{j = 0}^ N {\varphi (\ee^ {\ii t})^ j \overline{\varphi (\ee^ {\ii s})}^ j} \Big|^ 2 \dd m(\ee^ {\ii t}) \rd m (\ee^ {\ii s}).
		\end{align*}
		Since $C_{\varphi}^ * C_{\varphi} \in S_2 (H^ 2)$, $C_{\varphi}$ is compact, which means that $|\varphi (\ee^ {\ii \theta})| < 1$ for $m$-a.e. $\ee^ {\ii \theta} \in \TT$ (see \cite{Sch}). Therefore, the set
		\begin{align*}
			\left\{ (\ee^ {\ii t}, \ee^ {\ii s}) \in \TT \times \TT: \big| \varphi (\ee^ {\ii t}) \overline{\varphi (\ee^ {\ii s})} \big| = 1 \right\}
		\end{align*}
		is a null set with respect to the measure $\rd (m \times m)$. It follows that
		\begin{align*}
			\sum_{j, k = 0}^ N {|\langle C_{\varphi}^ * C_{\varphi} (z^ j), z^ k \rangle_2|^ 2} = \int_{\TT} \int_{\TT} \frac{\big| 1 - \varphi (\ee^ {\ii t})^ {N + 1} \overline{\varphi (\ee^ {\ii s})}^ {N + 1} \big|^ 2}{\big| 1 - \varphi (\ee^ {\ii t}) \overline{\varphi (\ee^ {\ii s})} \big|^ 2} \dd m(\ee^ {\ii t}) \rd m (\ee^ {\ii s}).
		\end{align*}
		Now applying Fatou's lemma, we find that
		\begin{align*}
			\int_{\TT} \int_{\TT} \frac{\rd m (\ee^ {\ii t}) \rd m (\ee^ {\ii s})}{\big| 1 - \varphi (\ee^ {\ii t}) \overline{\varphi (\ee^ {\ii s})} \big|^ 2} & \leqslant \varliminf_{N \rightarrow \infty} {\int_{\TT} \int_{\TT} \frac{\big| 1 - \varphi (\ee^ {\ii t})^ {N + 1} \overline{\varphi (\ee^ {\ii s})}^ {N + 1} \big|^ 2}{\big| 1 - \varphi (\ee^ {\ii t}) \overline{\varphi (\ee^ {\ii s})} \big|^ 2} \dd m (\ee^ {\ii t}) \rd m (\ee^ {\ii s})} \\
			& = \varliminf_{N \rightarrow \infty} {\sum_{j, k = 0}^ N {|\langle C_{\varphi}^ * C_{\varphi} (z^ j), z^ k \rangle_2|^ 2}} \\
			& = \| C_{\varphi}^ * C_{\varphi} \|_{S_2 (H^ 2)}^ 2.
		\end{align*}
		In other words, we derive that: $C_{\varphi}$ belongs to $S_4 (H^ 2)$ if and only if
		\begin{gather*}
			\int_{\TT} \int_{\TT} \frac{\rd m(\ee^ {\ii t}) \rd m (\ee^ {\ii s})}{\big| 1 - \varphi (\ee^ {\ii t}) \overline{\varphi (\ee^ {\ii s})} \big|^ 2} < +\infty. \qedhere
		\end{gather*}
	\end{remark}

	Our second criterion for the Hilbert--Schmidt class is analogous to Theorem \ref{proposition: sufficient condition for Cpsi*Cphi in Sp}, but under a weaker hypothesis.

	\begin{prop} \label{proposition: sufficient condition for S2, Hardy space, II}
		Suppose $\varphi$ and $\psi$ are analytic self-mappings of $\DD$. If
		\begin{align*}
			\int_{\DD} \frac{|\varphi' (z) \psi' (z)|}{(1 - |\varphi (z)|)^ \frac{3}{2} (1 - |\psi (z)|)^ \frac{3}{2}} \dd \sigma (z) < +\infty,
		\end{align*}
		then $C_{\psi}^ * C_{\varphi} \in S_2 (H^ 2)$. Here $\rd \sigma (z) = \log{(1 / |z|^ 2)} \dd A(z)$.
	\end{prop}

	To prove the above proposition, we require one more simple lemma. It should first be noted that every bounded operator $A$ on $L_a^ 2 (\rd \sigma)$ is an integral operator. To see this, letting $K_z$ denote the reproducing kernel for $L_a^ 2 (\rd \sigma)$ at $z \in \DD$, then we have
	\begin{align*}
		A f(z) = \langle A f, K_z \rangle_{L^ 2 (\DD, \rd \sigma)} = \langle f, A^ * K_z \rangle_{L^ 2 (\DD, \rd \sigma)} = \int_{\DD} f(w) \overline{A^ * K_z (w)} \dd \sigma (w).
	\end{align*}

	\begin{lem} \label{lemma: characterization of the Hilbert--Schmidt class}
		A bounded linear operator $A$ on $L_a^ 2 (\rd \sigma)$ belongs to $S_2 (L_a^ 2 (\rd \sigma))$ if and only if
		\begin{align*}
			\int_{\DD} \int_{\DD} \left| A^ * K_z (w) \right|^ 2 \dd \sigma (w) \rd \sigma (z) < +\infty,
		\end{align*}
		where $K_z$ is the reproducing kernel for $L_a^ 2 (\rd \sigma)$ at $z \in \DD$.
	\end{lem}

	\begin{proof}
		Consider the integral operator
		\begin{align*}
			\widetilde{A} f (z) = \int_{\DD} f(w) \overline{A^ * K_z (w)} \dd \sigma (w)
		\end{align*}
		defined on the Lebesgue space $L^ 2 (\DD, \rd \sigma)$. For every $f \in L^ 2 (\DD, \rd \sigma) \ominus L_a^ 2 (\rd \sigma)$, we have
		\begin{align*}
			\widetilde{A} f (z) = \langle f, A^ * K_z \rangle_{L^ 2 (\DD, \rd \sigma)} = 0, \qquad z \in \DD,
		\end{align*}
		since $A^ * K_z \in L_a^ 2 (\rd \sigma)$. Thus, under the orthogonal decomposition
		\begin{align*}
			L^ 2 (\DD, \rd \sigma) = L_a^ 2 (\rd \sigma) \oplus [L^ 2 (\DD, \rd \sigma) \ominus L_a^ 2 (\rd \sigma)],
		\end{align*}
		we obtain $\widetilde{A} = A \oplus 0$. We conclude that $\widetilde{A}$ is bounded on $L^ 2 (\DD, \rd \sigma)$, and $A \in S_2 (L_a^ 2 (\rd \sigma))$ if and only if $\widetilde{A} \in S_2 (L^ 2 (\DD, \rd \sigma))$. Now the assertion of the lemma follows from the standard result \cite[Theorem 3.5]{Zhu}.
	\end{proof}

	We are now in a position to prove Proposition \ref{proposition: sufficient condition for S2, Hardy space, II}.

	\begin{proof}[Proof of Proposition \ref{proposition: sufficient condition for S2, Hardy space, II}]
		Recall that $U$ is the unitary operator from $H^ 2$ onto $\CC\oplus L_a^ 2 (\rd \sigma)$ given by (\ref{for1}). Thus, $C_{\psi}^ * C_{\varphi} \in S_2 (H^ 2)$ if and only if $U C_{\psi}^ * C_{\varphi} U^ * \in S_2 (\CC \oplus L_a^ 2 (\rd \sigma))$. By (\ref{eq3}), we have
		\begin{align*}
			U C_{\psi}^ * C_{\varphi} U^ * =
			\begin{bmatrix}
				1 & 0 \\
				\Lambda_{\psi}^ * & D_{\psi}^ *
			\end{bmatrix}
			\begin{bmatrix}
				1 & \Lambda_{\varphi} \\
				0 & D_{\varphi}
			\end{bmatrix}
			=
			\begin{bmatrix}
				1 & \Lambda_{\varphi} \\
				\Lambda_{\psi}^ * & \Lambda_{\psi}^ * \Lambda_{\varphi} + D_{\psi}^ * D_{\varphi}
			\end{bmatrix}
			.
		\end{align*}
		Since the range of $\Lambda_{\varphi}$ and the domain of $\Lambda_{\psi}^ *$ are $\CC$, the operator
		\begin{align*}
			\begin{bmatrix}
				1 & \Lambda_{\varphi} \\
				\Lambda_{\psi}^ * & \Lambda_{\psi}^ * \Lambda_{\varphi}
			\end{bmatrix}
		\end{align*}
		is of finite rank, and in particular is in $S_2 (\CC \oplus L_a^ 2 (\rd \sigma))$. This observation implies that $U C_{\psi}^ * C_{\varphi} U^ * \in S_2 (\CC \oplus L_a^ 2 (\rd \sigma))$ if and only if the operator
		\begin{align*}
			\begin{bmatrix}
				0 & 0 \\
				0 & D_{\psi}^ * D_{\varphi}
			\end{bmatrix}
		\end{align*}
		is in $S_2 (\CC \oplus L_a^ 2 (\rd \sigma))$, which is equivalent to that $D_{\psi}^ * D_{\varphi} \in S_2 (L_a^ 2 (\rd \sigma))$. Therefore, the proof will be concluded once we show that
\(D_{\psi}^* D_{\varphi}\) belongs to the Hilbert--Schmidt class
\(S_2(L_a^2(d\sigma))\).

Letting $K_z$ be the reproducing kernel for $L_a^ 2 (\rd \sigma)$ at $z \in \DD$, we have
		\begin{align*}
			D_{\varphi}^ * h(z) & = \langle D_{\varphi}^ * h, K_z \rangle_{L^ 2 (\DD, \rd \sigma)} = \langle h, D_{\varphi} K_z \rangle_{L^ 2 (\DD, \rd \sigma)} \\
			& = \int_{\DD} h(w) \overline{\varphi' (w)} \overline{K_z (\varphi (w))} \dd \sigma (w)
		\end{align*}
		for $h \in L_a^ 2 (\rd \sigma)$.
		Hence,
		\begin{align*}
			D_{\varphi}^ * D_{\psi} K_z (w) = \int_{\DD} \psi' (u) \overline{\varphi' (u)} K_z (\psi (u)) \overline{K_w (\varphi (u))} \dd \sigma (u), \qquad z, w \in \DD.
		\end{align*}
		It follows from  Minkowski's integral inequality that
		\begin{align*}
			&\| D_{\varphi}^ * D_{\psi} K_z (w) \|_{L^ 2 (\DD \times \DD, \rd (\sigma \times \sigma) (z, w))}\\
 &\ \ \ \leqslant \int_{\DD} |\psi' (u) \varphi' (u)| \left\| K_z (\psi (u)) \overline{K_w (\varphi (u))} \right\|_{L^ 2 (\DD \times \DD, \rd (\sigma  \times \sigma) (z, w))} \! \! \! \! \dd \sigma (u).
		\end{align*}
		elementary calculations show that
		\begin{align*}
			\left\| K_z (\psi (u)) \overline{K_w (\varphi (u))} \right\|^ 2_{L^ 2 (\DD \times \DD, \rd (\sigma  \times \sigma) (z, w))} & = \int_{\DD} \frac{|1 + \overline{z} \psi (u)|^ 2}{|1 - \overline{z} \psi (u)|^ 6} \dd \sigma (z) \cdot \int_{\DD} \frac{|1 + \overline{w} \varphi (u)|^ 2}{|1 - \overline{w} \varphi (u)|^ 6} \dd \sigma (w) \\
			& \leqslant \frac{C}{(1 - |\psi (u)|)^ 3 (1 - |\varphi (u)|)^ 3},
		\end{align*}
		where $C > 0$ is a constant and the inequality  follows from  an argument analogous to the one used for (\ref{ineq10}). By Lemma \ref{lemma: characterization of the Hilbert--Schmidt class}, we conclude that $D_{\psi}^ * D_{\varphi} \in S_2 (L_a^ 2 (\rd \sigma))$.
	\end{proof}

	We now turn to the discussion of when  $C_{\psi}^*C_{\varphi}$ is a trace-class operator on $H^2$.
	
	\begin{prop} \label{proposition: sufficient condition for S1, Hardy space}
		Suppose $\varphi$ and $\psi$ are analytic self-mappings of $\DD$. Then $C_{\psi}^ * C_{\varphi} \in S_1 (H^ 2)$ if
		\begin{align}
			\int_{\TT} \frac{\rd m (\ee^ {\ii t})}{\sqrt{1 - |\varphi (\ee^ {\ii t})|} \sqrt{1 - |\psi (\ee^ {\ii t})|}} < +\infty. \label{ineq8}
		\end{align}
	\end{prop}

	\begin{proof}
		Assume  that (\ref{ineq8}) holds. Let $\{ e_n \}_{n = 1}^ {\infty}$ be any orthonormal  basis of $H^ 2$. Then we have
		\begin{align*}
			\sum_{n = 1}^ {\infty} {\left| \langle C_{\psi}^ * C_{\varphi} e_n, e_n \rangle_2 \right|} & = \sum_{n = 1}^ {\infty} {\left| \langle C_{\varphi} e_n, C_{\psi} e_n \rangle_2 \right|} \\
			& \leqslant \sum_{n = 1}^ {\infty} {\int_{\TT} |e_n (\varphi (\ee^ {\ii t}))| |e_n (\psi (\ee^ {\ii t}))| \dd m(\ee^ {\ii t})} \\
			& = \int_{\TT} \sum_{n = 1}^ {\infty} {|e_n (\varphi (\ee^ {\ii t}))| |e_n (\psi (\ee^ {\ii t}))|} \dd m(\ee^ {\ii t}) \\
			& \leqslant \int_{\TT} \left[ \sum_{n = 1}^ {\infty} {|e_n (\varphi (\ee^ {\ii t}))|^ 2} \right]^ {\frac{1}{2}} \left[ \sum_{n = 1}^ {\infty} {|e_n (\psi (\ee^ {\ii t}))|^ 2} \right]^ {\frac{1}{2}} \dd m(\ee^ {\ii t}).
		\end{align*}
		By Condition (\ref{ineq8}), $|\varphi (\ee^ {\ii \theta})|$ and $|\psi (\ee^ {\ii \theta})|$ are less than $1$ for $m$-a.e. $\ee^ {\ii \theta} \in \TT$. It follows from  Theorem \ref{theorem: reproducing kernels} that
		\begin{align*}
			\sum_{n = 1}^ {\infty} {\left| \langle C_{\psi}^ * C_{\varphi} e_n, e_n \rangle_2 \right|} \leqslant \int_{\TT} \frac{\rd m(\ee^ {\ii t})}{\sqrt{1 - |\varphi (\ee^ {\ii t})|^ 2} \sqrt{1 - |\psi (\ee^ {\ii t})|^ 2}} < +\infty,
		\end{align*}
		which yields that $C_{\psi}^ * C_{\varphi}$ belongs to $S_1 (H^ 2)$.
	\end{proof}

	\begin{remark}
		In the case that  $\varphi = \psi$,  by Proposition \ref{proposition: sufficient condition for S1, Hardy space} we  observe that $C_{\varphi}$ is in $S_2 (H^ 2)$ if
		\begin{align*}
			\int_{\TT} \frac{\rd m(\ee^ {\ii t})}{1 - |\varphi (\ee^ {\ii t})|} < +\infty.
		\end{align*}
		In fact,  Shapiro and  Taylor \cite[Theorem 3.1]{ST} showed that this condition for $C_{\varphi}$ to be in $S_2 (H^ 2)$ is not only sufficient but also necessary.
	\end{remark}
	
The next proposition gives a necessary condition for the operator $C_{\psi}^ * C_{\varphi}$  to be in the trace class on the Hardy space.

	\begin{prop} \label{proposition: trace formula, I}
		Suppose that $\varphi$ and $\psi$ are analytic self-mappings of $\DD$. If $C_{\psi}^ * C_{\varphi}$ belongs to $S_1 (H^ 2)$, then the limit
		\begin{align}
			\lim_{\DD \ni \zeta \rightarrow 1} {\int_{\TT} \frac{\rd m(\ee^ {\ii t})}{1 - \zeta \varphi (\ee^ {\ii t}) \overline{\psi (\ee^ {\ii t})}}} \label{for2}
		\end{align}
		exists. In this case, the above quantity is the trace of $C_{\psi}^ * C_{\varphi}$.
	\end{prop}

	\begin{proof}
		By  the definition of the trace of an operator, we have
		\begin{align*}
			\tr{(C_{\psi}^ * C_{\varphi})} & = \sum_{n = 0}^ {\infty} {\langle C_{\psi}^ * C_{\varphi} (z^ n), z^ n \rangle_2} = \sum_{n = 0}^ {\infty} {\langle \varphi^ n, \psi^ n \rangle_2} \\
			& = \lim_{\DD \ni \zeta \rightarrow 1} {\sum_{n = 0}^ {\infty} {\zeta^ n \langle \varphi^ n, \psi^ n \rangle_2}} \\
			& = \lim_{\DD \ni \zeta \rightarrow 1} {\int_{\TT} \frac{\rd m(\ee^ {\ii t})}{1 - \zeta \varphi (\ee^ {\ii t}) \overline{\psi (\ee^ {\ii t})}}}
		\end{align*}
for every $C_{\psi}^ * C_{\varphi}$ in $S_1 (H^ 2)$. Here,
		 the third ``$=$'' is due to the Lebesgue dominated convergence theorem, since
		\begin{align*}
			\sum_{n = 0}^ {\infty} {\left| \zeta^ n \langle \varphi^ n, \psi^ n \rangle_2 \right|} \leqslant \sum_{n = 0}^ {\infty} {\left| \langle C_{\psi}^ * C_{\varphi} (z^ n), z^ n \rangle_2 \right|} < +\infty
		\end{align*}
		for $\zeta$ in $\DD$. Therefore,  the limit (\ref{for2}) exists if $C_{\psi}^ * C_{\varphi}\in S_1 (H^ 2)$.
	\end{proof}

	Existing results do not provide a full characterization of the conditions under which $C_{\psi}^ * C_{\varphi}$ belongs to $S_p (H^2)$. Therefore, at the end of this subsection, we consider a special class of composition operators (those with finite Blaschke product symbols) and obtain a necessary and sufficient condition for such operators to be Hilbert--Schmidt. Our main tool is Theorem \ref{theorem: Aleksandrov operators}, and the following lemma provides the foundation for the subsequent calculations.
	
	\begin{lem} \label{lemma: Aleksandrov--Clark measure of B}
		Suppose $B$ is a finite Blaschke product of degree $n$. Let $\{ \rd \mu_{\alpha} : \alpha \in \TT \}$ be the family of the Aleksandrov--Clark measures associated with $B$. Then we have
		\begin{align*}
			\rd \mu_{\alpha} = \sum_{j = 1}^ n {\frac{1}{|B' (\zeta_j)|} \dd \delta_{\zeta_j}},
		\end{align*}
		where $\zeta_1, \zeta_2, \cdots, \zeta_n$ are the $n$ distinct roots of the equation $B(z) = \alpha$ and each $\rd \delta_{\zeta_j}$ is the point mass at $\zeta_j$.
	\end{lem}

	\begin{proof}
		According to (3) of \cite[Corollary 7.2]{Ros}, $\rd \mu_{\alpha}$ is singular with respect to $\rd m$ for each $\alpha \in \TT$. For every $\alpha \in \TT$,  we have
		\begin{align*}
			\frac{1 - |B(z)|^ 2}{|\alpha - B(z)|^ 2} = P_{B(z)} (\alpha) = \int_{\TT} P_z (\ee^ {\ii t}) \dd \mu_{\alpha} (\ee^ {\ii t}).
		\end{align*}
		Hence,  we obtain that
		\begin{align*}
			\lim_{r \rightarrow 1^ -} {\int_{\TT} P_{r \zeta} (\ee^ {\ii t}) \dd \mu_{\alpha} (\ee^ {\ii t})} = 0
		\end{align*}
		if $\zeta \in \TT \backslash \{ \zeta_1, \zeta_2, \dots, \zeta_n \}$. By (2) of \cite[Proposition 4.14]{Ros}, we have the inequality
		\begin{align*}
			(\underline{D} \mu_{\alpha}) (\zeta) \leqslant \varliminf_{r \rightarrow 1^ -} {\int_{\TT} P_{r \zeta} (\ee^ {\ii t}) \dd \mu_{\alpha} (\ee^ {\ii t})},
		\end{align*}
		where
		\begin{align*}
			(\underline{D} \mu_{\alpha}) (\ee^ {\ii \theta}) = \varliminf_{\delta \rightarrow 0^ +} {\frac{\mu_{\alpha} (\{ \ee^ {\ii t} \in \TT : |\theta - t| < \delta \})}{m (\{ \ee^ {\ii t} \in \TT : |\theta - t| < \delta \})}}
		\end{align*}
		is the lower symmetric derivative of $\rd \mu_{\alpha}$ at $\ee^ {\ii \theta} \in \TT$. Thus,
		\begin{align*}
			(\underline{D} \mu_{\alpha}) (\zeta) = 0, \qquad \zeta \in \TT \backslash \{ \zeta_1, \zeta_2, \cdots, \zeta_n \}.
		\end{align*}
		On the other hand, according to (5) of \cite[Proposition 4.14]{Ros}, the set
		\begin{align*}
			\left\{ \zeta \in \TT : (\underline{D} \mu_{\alpha}) (\zeta) = +\infty \right\}
		\end{align*}
		is a carrier for $\rd \mu_{\alpha}$, which means that
		\begin{align*}
			\mu_{\alpha} (E) = \mu_{\alpha} (E \cap \left\{ \zeta \in \TT : (\underline{D} \mu_{\alpha}) (\zeta) = +\infty \right\})
		\end{align*}
		for every Borel set $E$ in $\TT$. Therefore, $\{ \zeta_1, \zeta_2, \cdots, \zeta_n \}$ is a carrier for $\rd \mu_{\alpha}$, which yields  that $\rd \mu_{\alpha}$ is a linear combination of point masses $\rd \zeta_1, \rd \zeta_2, \cdots, \rd \zeta_n$.
		The assertion now follows from \cite[Corollary 9.2]{Ros}.
	\end{proof}

Now we can establish a characterization for $C_B^ * C_{\varphi}$ to be in  $S_2 (H^ 2)$, where $B$ is  a finite Blaschke product.

	\begin{thm} \label{theorem: characterization of CB*Cpsi in S2}
		Suppose that $B$ is a finite Blaschke product and $\varphi$ is an analytic self-mapping of $\DD$. Then $C_B^ * C_{\varphi} \in S_2 (H^ 2)$ if and only if $C_{\varphi} \in S_2 (H^ 2)$.
	\end{thm}

	\begin{proof}
		Since $S_2 (H^ 2)$ is a two-sided ideal, the ``if'' part is trivial. To show the ``only if'' part, we first apply Theorem \ref{theorem: Aleksandrov operators} to obtain
		\begin{align*}
			C_B^ * C_{\varphi} (z^ j) (\alpha) = \int_{\TT} \varphi (\ee^ {\ii t})^ j \dd \mu_{\alpha} (\ee^ {\ii t}), \qquad \textrm{$m$-a.e. } \alpha \in \TT,
		\end{align*}
		where $j$ is a nonnegative integer and $\rd \mu_{\alpha}$ is the Aleksandrov--Clark measure associated with $B$ at $\alpha$. As a consequence, we have
		\begin{align*}
			\sum_{j = 0}^ {\infty} {\| C_B^ * C_{\varphi} (z^ j) \|_2^ 2 r^ {2 j}} & = \sum_{j = 0}^ {\infty} {\int_{\TT} \int_{\TT} \int_{\TT} \varphi (\ee^ {\ii t})^ j \overline{\varphi (\ee^ {\ii s})^ j} \dd \mu_{\alpha} (\ee^ {\ii t}) \rd \mu_{\alpha} (\ee^ {\ii s}) \rd m(\alpha)} r^ {2 j} \\
			& = \int_{\TT} \int_{\TT} \int_{\TT} \frac{\rd \mu_{\alpha} (\ee^ {\ii t}) \rd \mu_{\alpha} (\ee^ {\ii s})}{1 - r^ 2 \varphi (\ee^ {\ii t}) \overline{\varphi (\ee^ {\ii s})}} \dd m(\alpha)
		\end{align*}
		for each $r \in (0, 1)$.
		
		\begin{claim}
			For each $r \in (0, 1)$, the inequality
			\begin{align*}
				\int_{\TT} \int_{\TT} \frac{\rd \mu_{\alpha} (\ee^ {\ii t}) \rd \mu_{\alpha} (\ee^ {\ii s})}{1 - r^ 2 \varphi (\ee^ {\ii t}) \overline{\varphi (\ee^ {\ii s})}} \geqslant \varepsilon \int_{\TT} \frac{\rd \mu_{\alpha} (\ee^ {\ii t})}{1 - r^ 2 |\varphi (\ee^ {\ii t})|^ 2}
			\end{align*}
			holds, where $\varepsilon > 0$ is a constant independent of $r$.
		\end{claim}
	
		Let us first assume that the above claim holds. Then it follows that
		\begin{align*}
			\sum_{j = 0}^ {\infty} {\| C_B^ * C_{\varphi} (z^ j) \|_2^ 2} \geqslant \varepsilon \int_{\TT} \int_{\TT} \frac{\rd \mu_{\alpha} (\ee^ {\ii t})}{1 - r^ 2 |\varphi (\ee^ {\ii t})|^ 2} \dd m(\alpha) = \varepsilon \int_{\TT} \frac{\rd m (\ee^ {\ii t})}{1 - r^ 2 |\varphi (\ee^ {\ii t})|^ 2},
		\end{align*}
		where the ``$=$'' is due to Aleksandrov's disintegration theorem (Theorem \ref{theorem: Aleksandrov's disintegration theorem}). Letting $r \rightarrow 1^ -$, we deduce by the monotone convergence theorem that
		\begin{align*}
			\int_{\TT} \frac{\rd m(\ee^ {\ii t})}{1 - |\varphi (\ee^ {\ii t})|^ 2} \leqslant \varepsilon^ {-1} \sum_{j = 0}^ {\infty} {\| C_B^ * C_{\varphi} (z^ j) \|_2^ 2} < +\infty,
		\end{align*}
		which gives $C_{\varphi} \in S_2 (H^ 2)$ by \cite[Theorem 3.1]{ST}.
		
		In order to complete the proof, it remains to prove the claim. Adopting the notation in Lemma \ref{lemma: Aleksandrov--Clark measure of B}, we deduce  that
		\begin{align*}
			\int_{\TT} \int_{\TT} \frac{\rd \mu_{\alpha} (\ee^ {\ii t}) \rd \mu_{\alpha} (\ee^ {\ii s})}{1 - r^ 2 \varphi (\ee^ {\ii t}) \overline{\varphi (\ee^ {\ii s})}} = \sum_{j, k = 1}^ n {\frac{1}{|B' (\zeta_j) B' (\zeta_k)|} \frac{1}{1 - r^ 2 \varphi (\zeta_j) \overline{\varphi (\zeta_k)}}}.
		\end{align*}
		Noting that
		\begin{align*}
			\frac{1}{1 - a} + \frac{1}{1 - \overline{a}} = \frac{2 \re{(1 - a)}}{|1 - a|^ 2} \geqslant 0
		\end{align*}
for $a \in \DD$,  we have
		\begin{align*}
			\int_{\TT} \int_{\TT} \frac{\rd \mu_{\alpha} (\ee^ {\ii t}) \rd \mu_{\alpha} (\ee^ {\ii s})}{1 - r^ 2 \varphi (\ee^ {\ii t}) \overline{\varphi (\ee^ {\ii s})}} \geqslant \sum_{j = 1}^ n {\frac{1}{|B' (\zeta_j)|^ 2} \frac{1}{1 - r^ 2 |\varphi (\zeta_j)|^ 2}}.
		\end{align*}
		Since $B'$ is continuous on $\TT$, there exists an $\varepsilon > 0$ such that $|B' (\ee^ {\ii \theta})|^ {-1} \geqslant \varepsilon$ for all $\ee^ {\ii \theta} \in \TT$. We hence obtain
		\begin{align*}
			\int_{\TT} \int_{\TT} \frac{\rd \mu_{\alpha} (\ee^ {\ii t}) \rd \mu_{\alpha} (\ee^ {\ii s})}{1 - r^ 2 \varphi (\ee^ {\ii t}) \overline{\varphi (\ee^ {\ii s})}} \geqslant \varepsilon \sum_{j = 1}^ n {\frac{1}{|B' (\zeta_j)|} \frac{1}{1 - r^ 2 |\varphi (\zeta_j)|^ 2}} = \varepsilon \int_{\TT} \frac{\rd \mu_{\alpha} (\ee^ {\ii t})}{1 - r^ 2 |\varphi (\ee^ {\ii t})|^ 2},
		\end{align*}
		as required.
	\end{proof}

	\subsection{\texorpdfstring{The operator $C_{\varphi} C_{\psi}^ *$}{The Operator Cpsi* Cphi}}\mbox{}

Our analysis of $C_{\psi}^* C_{\varphi}$, as shown previously, depends mainly on the inner product identity
	\begin{align*}
		\langle C_{\psi}^ * C_{\varphi} f, g \rangle_2 = \langle C_{\varphi} f, C_{\psi} g \rangle_2 = \langle f \circ \varphi, g \circ \psi \rangle_2,
	\end{align*}
	which forms the basis for the calculations that follow. In contrast to the previous case, a different treatment is necessary for $C_{\varphi} C_{\psi}^ *$ because $C_{\varphi}^* f$ is generally difficult to compute directly. Recall that $D_{\varphi}$ is the bounded linear operator $$f \longmapsto \varphi' \cdot (f \circ \varphi)$$ defined on $L_a^ 2 (\rd \sigma)$. Moreover, by a similar argument as the one used in the first paragraph of the proof of Proposition \ref{proposition: sufficient condition for S2, Hardy space, II}, one can show that $C_{\varphi} C_{\psi}^ *$ belongs to $S_p (H^ 2)$ if and only if $D_{\varphi} D_{\psi}^ *$ belongs to $S_p (L_a^ 2 (\rd \sigma))$ for each $0 < p < \infty$. Also observe by Lemma \ref{lemma: Cphi* Cphi is a Toeplitz opeartor} that
	\begin{align*}
		\left( D_{\varphi} D_{\psi}^ * \right)^ * \left( D_{\varphi} D_{\psi}^ * \right) = D_{\psi} D_{\varphi}^ * D_{\varphi} D_{\psi}^ * = D_{\psi} T_{\Phi} D_{\psi}^*,
	\end{align*}
	where $$\Phi (z) = \frac{N_{\varphi} (z)}{\log{(1 / |z|)}}$$ and $T_{\Phi}$ is the corresponding  Toeplitz operator on $L_a^ 2 (\rd \sigma)$. Therefore, it further holds that $C_{\varphi} C_{\psi}^ * \in S_p (H^ 2)$ if and only if $D_{\psi} T_{\Phi} D_{\psi}^ * \in S_{\frac{p}{2}} (L_a^ 2 (\rd \sigma))$, where $0 < p < \infty$.

	\begin{thm} \label{proposition: sufficient condition for Cphi Cpsi* in Sp}
		Suppose that $2\leqslant  p < \infty$, $\varphi$ and $\psi$ are analytic self-mappings of $\DD$. Then  $C_{\varphi} C_{\psi}^ * \in S_p (H^ 2)$ if
		\begin{align} \label{ineq2}
			\int_{\DD} \int_{\DD} \bigg[ \frac{\left| 1 + \overline{z} w \right|}{\left| 1 - \overline{z} w \right|^ 3} \sqrt{N_{\varphi} (z) N_{\psi} (w)} \bigg]^ p \dd A(w) \rd A(z) < +\infty.
		\end{align}
	\end{thm}

	\begin{proof}
		Assume  that (\ref{ineq2}) holds and denote $\Phi (z) = N_{\varphi} (z) / \log{(1 / |z|)}$. According to the discussion in the first paragraph of this subsection, it suffices to show that $D_{\psi} T_{\Phi} D_{\psi}^ *$ belongs to $S_{\frac{p}{2}} (L_a^ 2 (\rd \sigma))$.
		Let $\{ e_n \}_{n = 1}^ {\infty}$ be an arbitrary orthonormal basis for $L_a^ 2 (\rd \sigma)$. Then by Lemma \ref{lemma: Cphi* Cphi is a Toeplitz opeartor} we have
		\begin{align*}
			\sum_{n = 1}^ {\infty} {|\langle D_{\psi} T_{\Phi} D_{\psi}^* e_n, e_n \rangle_{L^2 (\DD, \rd \sigma)}|^ {\frac{p}{2}}} & = \sum_{n = 1}^ {\infty} {|\langle T_{\Phi} D_{\psi}^ * e_n, D_{\psi}^ * e_n \rangle_{L^ 2 (\DD, \rd \sigma)}|^ {\frac{p}{2}}} \\
			& = 2^ {\frac{p}{2}} \sum_{n = 1}^ {\infty} {\Big[ \int_{\DD} |D_{\psi}^ * e_n (z)|^ 2 N_{\varphi} (z) \dd A (z) \Big]^ {\frac{p}{2}}} \\
			& \leqslant 2^ {\frac{p}{2}} \sum_{n = 1}^ {\infty} {\int_{\DD} |D_{\psi}^ * e_n (z)|^ {p} N_{\varphi} (z)^ {\frac{p}{2}} \dd A(z)}.
		\end{align*}
		Using the elementary inequality
		\begin{align*}
			\Big( \sum_{n = 1}^ {\infty} {|c_n|} \Big)^ q \leqslant  \sum_{n = 1}^\infty {|c_n|^ q} \ \ \ \ \ \big(c_n \in \CC,\ 0 < q\leqslant 1\big),
		\end{align*}
		we obtain that
		\begin{align*}
			\sum_{n = 1}^ {\infty} {|D_{\psi}^ * e_n (z)|^ p} & \leqslant \Big[ \sum_{n = 1}^ {\infty} {|D_{\psi}^ * e_n (z)|^ 2} \Big]^ {\frac{p}{2}} = \Big[ \sum_{n = 1}^ {\infty} |\langle e_n, D_{\psi} K_z \rangle_{L^ 2 (\DD, \rd \sigma)}|^ 2 \Big]^ {\frac{p}{2}} = \| D_{\psi} K_z \|_{L^ 2 (\DD, \rd \sigma)}^ p,
		\end{align*}
		where $K_z$ is the reproducing kernel for $L_a^ 2 (\rd \sigma)$ at $z$.  Furthermore, we have by Theorem \ref{theorem: a change of variables} that
		\begin{align*}
			\| D_{\psi} K_z \|_{L^ 2 (\DD, \rd \sigma)}^ 2 & = \int_{\DD} |\psi' (w)|^ 2 \frac{|1 + \overline{z} \psi (w)|^ 2}{|1 - \overline{z} \psi (w)|^ 6} \dd \sigma (w) \\
			& = 2 \int_{\DD} \frac{|1 + \overline{z} w|^ 2}{|1 - \overline{z} w|^ 6} N_{\psi} (w) \dd A(w),
		\end{align*}
	which yields that
		\begin{align*}
			\sum_{n = 1}^ {\infty} {|\langle D_{\psi} T_{\Phi} D_{\psi}^ * e_n, e_n \rangle_{L^ 2 (\DD, \rd \sigma)}|^ {\frac{p}{2}}} & \leqslant 2^ {\frac{p}{2}} \int_{\DD} \bigg[ 2 \int_{\DD} \frac{|1 + \overline{z} w|^ 2}{|1 - \overline{z} w|^ 6} N_{\psi} (w) \dd A(w) \bigg]^ {\frac{p}{2}} N_{\varphi} (z)^ {\frac{p}{2}} \dd A(z) \\
			& \leqslant 2^ p \int_{\DD} \int_{\DD} \frac{\left| 1 + \overline{z} w \right|^ p}{|1 - \overline{z} w|^ {3 p}} N_{\varphi} (z)^ {\frac{p}{2}} N_{\psi} (w)^ {\frac{p}{2}} \dd A(w) \rd A(z) \\
			& < +\infty.
		\end{align*}
		Now we conclude that $D_{\psi} T_{\Phi} D_{\psi}^ *$ belongs to $S_{\frac{p}{2}} (L_a^ 2 (\rd \sigma))$, as desired.
	\end{proof}

	\begin{remark} \label{remark: proposition: sufficient condition for S2, Hardy space}
		Condition (\ref{ineq2}) is not only sufficient but also necessary in the case that $p = 2$. Indeed,  letting $p = 2$ in the proof of Theorem  \ref{proposition: sufficient condition for Cphi Cpsi* in Sp}, we get that
		\begin{align*}
			\sum_{n = 1}^ {\infty} {\left| \langle D_{\psi} T_{\Phi} D_{\psi}^ * e_n, e_n \rangle_{L^ 2 (\DD, \rd \sigma)} \right|^ 2} = 4 \int_{\DD} \int_{\DD} \frac{|1 + \overline{z} w|^ 2}{|1 - \overline{z} w|^ 6} N_{\varphi} (z) N_{\psi} (w) \dd A(w) \rd A(z).
		\end{align*}
		Then the necessity follows immediately.
	\end{remark}

	Our next objective is to derive a trace formula for $C_{\varphi} C_{\psi}^ *$. For this purpose, we need to establish some preliminary lemmas first. For a bounded operator $T$ on the Hardy space $H^ 2$, let
	\begin{align*}
		\mathcal{B} T (z) = \langle T k_z, k_z \rangle_2, \qquad z \in \DD,
	\end{align*}
	where $$k_z (\ee^ {\ii \theta}) = \frac{\sqrt{1 - |z|^ 2}}{1 - \overline{z} \ee^ {\ii \theta}}$$ is the normalized reproducing kernel for $H^ 2$ at $z$. The function $\mathcal{B} T$ is called the \textit{Berezin transform} of $T$.
	
	\begin{lem} \label{lemma: trace formula for positive operators}
		Suppose that $T$ is  a  bounded positive operator on the Hardy space $H^ 2$. Then $T$ belongs to $S_1 (H^ 2)$ if and only if the limit
		\begin{align}
			\lim_{r \rightarrow 1^ -} {\frac{1}{1 - r^ 2} \int_{\TT} \mathcal{B} T (r \ee^ {\ii t}) \dd m(\ee^ {\ii t})} \label{for3}
		\end{align}
		exists. In this case, \textup{(\ref{for3})} is the trace of $T$.
	\end{lem}

	\begin{proof}
		Note that $T$ belongs to $S_1 (H^ 2)$ if and only if $\sqrt{T}$ belongs to $S_2 (H^ 2)$. It follows from Theorem \ref{theorem: characterization of S2} that $T \in S_1 (H^ 2)$ if and only if
		\begin{align*}
			\sum_{n = 1}^ {\infty} {\big\| \sqrt{T} e_n \big\|_2^ 2} < +\infty
		\end{align*}
		for some orthonormal basis  $\{ e_n \}_{n = 1}^ {\infty}$ of $H^ 2$.

Noting that
		\begin{align*}
			\big\| \sqrt{T} e_n \big\|_2^ 2 = \lim_{r \rightarrow 1^ -} {\int_{\TT} \big| \sqrt{T} e_n (r \ee^ {\ii t}) \big|^ 2 \dd m(\ee^ {\ii t})} = \lim_{r \rightarrow 1^ -} {\int_{\TT} \big| \big\langle \sqrt{T} e_n, K_{r \ee^ {\ii t}} \big\rangle_2 \big|^ 2 \dd m(\ee^ {\ii t})},
		\end{align*}
		we have
		\begin{align*}
			\sum_{n = 1}^ {\infty} {\big\| \sqrt{T} e_n \big\|_2^ 2} & = \sum_{n = 1}^ {\infty} {\lim_{r \rightarrow 1^ -} {\int_{\TT} \big| \big\langle \sqrt{T} e_n, K_{r \ee^ {\ii t}} \big\rangle_2 \big|^ 2 \dd m(\ee^ {\ii t})}} \\
			& = \lim_{r \rightarrow 1^ -} {\sum_{n = 1}^ {\infty} {\int_{\TT} \big| \big\langle \sqrt{T} e_n, K_{r \ee^ {\ii t}} \big\rangle_2 \big|^ 2 \dd m(\ee^ {\ii t})}} \\
			& = \lim_{r \rightarrow 1^ -} {\int_{\TT} \sum_{n = 1}^ {\infty} {\big| \big\langle e_n, \sqrt{T} K_{r \ee^ {\ii t}} \big\rangle_2 \big|^ 2} \dd m(\ee^ {\ii t})} \\
			& = \lim_{r \rightarrow 1^ -} {\int_{\TT}\big\| \sqrt{T} K_{r \ee^ {\ii t}} \big\|_2^ 2 \dd m(\ee^ {\ii t})} \\
			& = \lim_{r \rightarrow 1^ -} {\frac{1}{1 - r^ 2} \int_{\TT} \mathcal{B} T (r \ee^ {\ii t}) \dd m(\ee^ {\ii t})}.
		\end{align*}
		Here, the second ``$=$'' above is due to the monotone convergence theorem, since the integral
		\begin{align*}
			\int_{\TT} \big| \big\langle \sqrt{T} e_n, K_{r \ee^ {\ii t}} \big\rangle_2 \big|^ 2 \dd m(\ee^ {\ii t}) = \int_{\TT} \big| \sqrt{T} e_n (r \ee^ {\ii t}) \big|^ 2 \dd m(\ee^ {\ii t})
		\end{align*}
		is increasing in $r\in [0, 1)$. Therefore, $T$ belongs to $S_1 (H^ 2)$ if and only if the limit (\ref{for3}) exists. In this case, by Theorem \ref{theorem: characterization of S2} we also have
		\begin{gather*}
			\lim_{r \rightarrow 1^ -} {\frac{1}{1 - r^ 2} \int_{\TT} \mathcal{B} T (r \ee^ {\ii t}) \dd m(\ee^ {\ii t})} = \big\| \sqrt{T} \big\|_{S_2 (H^ 2)}^ 2 = \| T \|_{S_1 (H^ 2)} = \tr{(T)}.
		\end{gather*}
This finishes the proof of the lemma.
	\end{proof}

	The next lemma extends the trace formula in Lemma \ref{lemma: trace formula for positive operators} to all operators in $S_1 (H^ 2)$.
	
	\begin{lem} \label{lemma: trace formula}
		Suppose that $T$ is a bounded linear operator on $H^ 2$. If $T$ belongs to $S_1 (H^ 2)$, then
		\begin{align}
			\lim_{r \rightarrow 1^ -} {\frac{1}{1 - r^ 2} \int_{\TT} \mathcal{B} T (r \ee^ {\ii t}) \dd m(\ee^ {\ii t})} \label{for4}
		\end{align}
		exists. In this case, \textup{(\ref{for4})} is the trace of $T$.
	\end{lem}

	\begin{proof}
		Let us first consider the case that $T$ is self-adjoint. In this case, both $(|T| + T) / 2$ and $(|T| - T) / 2$ are positive and in $S_1 (H^ 2)$. Hence, according to Lemma \ref{lemma: trace formula for positive operators}, the limits
		\begin{align*}
			\lim_{r \rightarrow 1^ -} {\frac{1}{1 - r^ 2} \int_{\TT} \mathcal{B} \Big( \frac{|T| \pm T}{2} \Big) (r \ee^ {\ii t}) \dd m(\ee^ {\ii t})}
		\end{align*}
		exist and equal $\tr{\left[ (|T| \pm T) / 2 \right]}$. It follows that
		\begin{align*}
			\tr{(T)} = \lim_{r \rightarrow 1^ -} {\frac{1}{1 - r^ 2} \int_{\TT} \mathcal{B} T (r \ee^ {\ii t}) \dd m(\ee^ {\ii t})}.
		\end{align*}
		Now the assertion for general operators $T$ follows from the standard decomposition
		\begin{gather*}
			T = \frac{T + T^ *}{2} + \ii \frac{T - T^ *}{2 \ii},
		\end{gather*}
finishing the proof.
	\end{proof}

	\begin{prop} \label{proposition: trace formula, II}
		Suppose $\varphi$ and $\psi$ are analytic self-mappings of $\DD$. If $C_{\varphi} C_{\psi}^ *$ belongs to $S_1 (H^ 2)$, then the limit
		\begin{align*}
			\lim_{r \rightarrow 1^ -} {\int_{\TT} \frac{\rd m(\ee^ {\ii t})}{1 - \varphi (r \ee^ {\ii t}) \overline{\psi (r \ee^ {\ii t})}}}
		\end{align*}
		exists. In this case, the above quantity is the trace of $C_{\varphi} C_{\psi}^ *$.
	\end{prop}

	\begin{proof}
		First, note that
		\begin{align*}
			\left\langle C_{\varphi}^ * K_z, f \right\rangle_2 & = \left\langle K_z, C_{\varphi} f \right\rangle_2 = \overline{f (\varphi (z))} = \left\langle K_{\varphi (z)}, f \right\rangle_2,\ \ \  \ f\in H^2.
		\end{align*}
	This gives us that $C_{\varphi}^ * K_z = K_{\varphi (z)}$, and similarly $C_{\psi}^ * K_z = K_{\psi (z)}$.

		Then the  Berezin transform of $C_{\varphi} C_{\psi}^ *$ is given by
		\begin{align*}
			\mathcal{B} \left( C_{\varphi} C_{\psi}^ * \right) (z) & = \left\langle C_{\varphi} C_{\psi}^ * k_z, k_z \right\rangle_2 = \left\langle C_{\psi}^ * k_z, C_{\varphi}^ * k_z \right\rangle_2 = (1 - |z|^ 2) \left\langle K_{\psi (z)}, K_{\varphi (z)} \right\rangle_2 \\
			& = (1 - |z|^ 2) K_{\psi (z)} \left( \varphi (z) \right).
		\end{align*}
		 It follows from  Lemma \ref{lemma: trace formula} that the limit
		\begin{align*}
			\lim_{r \rightarrow 1^ -} {\frac{1}{1 - r^ 2} \int_{\TT} \mathcal{B} \left( C_{\varphi} C_{\psi}^ * \right) (r \ee^ {\ii t}) \dd m(\ee^ {\ii t})} = \lim_{r \rightarrow 1^ -} {\int_{\TT} \frac{1}{1 - \varphi (r \ee^ {\ii t}) \overline{\psi (r \ee^ {\ii t})}} \dd m(\ee^ {\ii t})}
		\end{align*}
		exists and is equal to the trace of $C_{\varphi} C_{\psi}^*$.
	\end{proof}

	\begin{remark}
		The similarity between the trace formulas for $C_{\psi}^ * C_{\varphi}$ and $C_{\varphi} C_{\psi}^ *$ (see Propositions \ref{proposition: trace formula, I} and \ref{proposition: trace formula, II}) is to be expected, as generally we have
		\begin{align*}
			\tr{(A B)} = \tr{(B A)},
		\end{align*}
		provided that $A \in S_p$, $B \in S_q$, and $p^ {-1} + q^ {-1} = 1$ (see \cite[Lemma 1.36]{Zhu} if needed).
	\end{remark}

	Recall that  we obtained  a criterion for $C_B^ * C_{\varphi}$ to be a Hilbert--Schmidt operator at the end of the previous subsection (Theorem \ref{theorem: characterization of CB*Cpsi in S2}),   where $B$ is a finite Blaschke product. However, it becomes much easier if we exchange $C_{B}^*$ and $C_{\varphi}$. More precisely, for any inner function $\psi$ and any $p \in (0, \infty)$, the characterization for $C_{\varphi} C_{\psi}^* \in S_p(H^2)$ follows directly.
	
	\begin{thm} \label{theorem: inner function symbols}
		Suppose that $0 < p < \infty$, $\varphi$ and $\psi$ are analytic self-mappings of $\DD$, and $\psi$ is an inner function. Then $C_{\varphi} C_{\psi}^ *$ is in $S_p (H^ 2)$ if and only if $C_{\varphi}$ is in $S_p (H^ 2)$.
	\end{thm}

	\begin{proof}
		The ``if'' part is obvious since $S_p (H^ 2)$ is a two-sided ideal. The ``only if'' part is a consequence of \cite[Theorem 1]{Nor}, which states that $C_{\psi}$ (with $\psi$ inner) is bounded below on $H^2$. Thus, $C_{\psi}^ *$ is surjective with a bounded right inverse $A$. Using the fact that $S_p (H^ 2)$ is a two-sided ideal again, we deduce that $C_{\varphi} = C_{\varphi} C_{\psi}^ * A$ is also in $S_p (H^ 2)$, as desired.
	\end{proof}

	\section{Composition operators on the Bergman space} \label{composition operators on the Bergman space}
	
	In the final section, we extend the results obtained on the Hardy space $H^2$ to the Bergman space $L_a^2$. Given the similarity of the techniques, we omit the proofs unless specific differences in the treatment of details arise. Throughout this section, let $$\sigma_n (z)=\sqrt{n + 1} z^ n, \ \ \ \ n=0, 1, 2, \cdots.$$  Thus $\{ \sigma_n \}_{n = 0}^ {\infty}$ forms an orthonormal basis for $L_a^ 2$.
	
	Let us first investigate composition operators with linear fractional  symbols. We begin with a lemma analogous to Theorem \ref{theorem: Cowen's theorem}.
	
	\begin{lem} \label{lemma: representation of the adjoint}
		Let $\psi (z) = \frac{a z + b}{c z + d}$ be a linear fractional map from $\DD$ into itself with $a d - b c \neq 0$. Then $\tau (z) = \frac{\overline{a} z - \overline{c}}{-\overline{b} z + \overline{d}}$ maps $\DD$ into itself, $g(z) = \frac{1}{-\overline{b} z + \overline{d}}$ and $h(z) = c z + d$ are bounded on $\DD$, and
		\begin{align*}
			C_{\psi}^ * = T_{g^ 2} C_{\tau} T_{h^ 2}^ *
		\end{align*}
on the Bergman space $L_a^2$.
	\end{lem}

	\begin{proof}
		The proof is the same as that of Theorem \ref{theorem: Cowen's theorem}, see \cite[Theorem 9.2]{CM} for the details. Let $K_z$ denote the reproducing kernel of $L_a^ 2$ at $z \in \DD$. Then a straightforward calculation gives
		\begin{align*}
			C_{\psi}^ * K_z = T_{g^ 2} C_{\tau} T_{h^ 2}^ * K_z, \qquad z \in \DD.
		\end{align*}
		Since $\spa{\{ K_z : z \in \DD \}}$ is dense in $L_a^ 2$, we obtain the desired identity.
	\end{proof}

	On the Bergman space $L_a^2$, we can also establish a result analogous to Lemma \ref{L-CZ}:
	
	\begin{lem} \label{lemma: compactness for linear fractional maps}
		Suppose $\varphi$ and $\psi$ are linear fractional maps from $\DD$ into itself. Let $\tau$ be the map associated with $\psi$ as described in Lemma \ref{lemma: representation of the adjoint}. Then on the Bergman space $L_a^2$:
		\begin{enumerate}[label = \textup{(\alph*)}]
			\item $C_{\varphi} C_{\psi}^ *$ is compact  if and only if $C_{\varphi} C_{\tau}$ is compact, and
			\item $C_{\psi}^ * C_{\varphi}$ is compact  if and only if $C_{\tau} C_{\varphi}$ is compact.
		\end{enumerate}
	\end{lem}

	To prove this lemma, we need some preliminaries.  Let $Q$ denote the difference quotient operator on $L_a^ 2$:
	\begin{align*}
		Q f (z) = \frac{f(z) - f(0)}{z}, \qquad z \in \DD.
	\end{align*}
	It is easy to see that $Q$ is bounded.
	
	\begin{lem} \label{lemma: operator identities}
		Suppose $\varphi$ is an analytic self-mapping of $\DD$ and $u$ is a bounded analytic function on $\DD$. Let
		\begin{align*}
			\alpha = Q \varphi \qquad \textrm{and} \qquad \beta = Q \alpha,
		\end{align*}
		which are bounded and analytic on $\DD$. Then on the Bergman space $L_a^ 2$:
		\begin{enumerate}[label = \textup{(\alph*)}]
			\item $C_{\varphi} T_u = T_{u \circ \varphi} C_{\varphi}$;
			\item $Q C_{\varphi} = T_{\alpha} C_{\varphi} Q$ and $Q T_{\alpha} = T_{\beta} + \alpha (0) Q$ if $\varphi(0)=0$.
		\end{enumerate}
	\end{lem}

	\begin{proof}
		Let $f$ be any function in $L_a^ 2$. Since
		\begin{align*}
			C_{\varphi} T_u f = (u \circ \varphi) (f \circ \varphi) = T_{u \circ \varphi} C_{\varphi} f,
		\end{align*}
		(a) holds.

		Using
		\begin{align*}
			Q C_{\varphi} f (z) = \frac{f(\varphi(z)) - f (0)}{z} = \frac{f(\varphi (z)) - f(0)}{\varphi (z)} \frac{\varphi (z)}{z} = T_{\alpha} C_{\varphi} Q f(z), \qquad z \in \DD,
		\end{align*}
		we obtain $Q C_{\varphi} = T_{\alpha} C_{\varphi} Q$. Finally, observing that
		\begin{align*}
			Q T_{\alpha} f (z) & = \frac{\alpha (z) f(z) - \alpha (0) f(0)}{z} = \frac{[\alpha (z) - \alpha (0)] f(z)}{z} + \alpha (0) \frac{f(z) - f(0)}{z} \\
			& = T_{\beta} f (z) + \alpha (0) Q f (z), \qquad z \in \DD,
		\end{align*}
		we get that $Q T_{\alpha} = T_{\beta} + \alpha (0) Q$. This gives (b).
	\end{proof}

	Now we are going to prove Lemma \ref{lemma: compactness for linear fractional maps}.

	\begin{proof}[Proof of Lemma \ref{lemma: compactness for linear fractional maps}]
		Following the argument for the corresponding result in the  Hardy space setting, (a) and the necessity part of (b) can be established similarly (see \cite{CZ}) and are omitted. It remains to prove the sufficiency part of (b). Replacing $\varphi$ by $\frac{\varphi (0) - \varphi}{1 - \overline{\varphi (0)} \varphi}$ if necessary, we may assume that $\varphi (0) = 0$. And since the conclusion is trivial if $\psi$ is a constant function, we further assume that $\psi$ is non-constant.
		
		Adopting notation in Lemma \ref{lemma: representation of the adjoint}, we get that
		\begin{align*}
			C_{\psi}^ * C_{\varphi} = T_{g^ 2} C_{\tau} T_{h^ 2}^ * C_{\varphi} = T_{g^ 2} C_{\tau} (\overline{c} T_{\overline{z}} + \overline{d} I)^ 2 C_{\varphi},
		\end{align*}
		where $I$ is the identity operator on $L_a^ 2$. For any
		\begin{align*}
			f(z) = a_0 + a_1 z + a_2 z^ 2 + \cdots
		\end{align*}
		in $L_a^ 2$, we have
		\begin{align*}
			T_{\overline{z}} \left( a_0 + a_1 z + a_2 z^ 2 + \cdots \right) & = \frac{1}{2} a_1 + \frac{2}{3} a_2 z + \frac{3}{4} a_3 z^ 2 + \cdots \\
			& = \left( a_1 + a_2 z + a_3 z^ 2 + \cdots \right) - \left( \frac{1}{2} a_1 + \frac{1}{3} a_2 z + \frac{1}{4} a_3 z^ 2 + \cdots \right).
		\end{align*}
		Hence,
		\begin{align*}
			T_{\overline{z}} = Q + R,
		\end{align*}
		where $R$ is the compact operator defined by
		\begin{align*}
			R f = \sum_{j = 0}^ {\infty} {\frac{1}{\sqrt{(j + 1) (j + 2)}} (\sigma_j \otimes \sigma_{j + 1}) f}, \ \ \ \ f\in L_a^2.
		\end{align*}
		It follows that
		\begin{align*}
			(\overline{c} T_{\overline{z}} + \overline{d} I)^ 2 = (\overline{c} Q + \overline{c} R + \overline{d} I)^ 2 = (\overline{c} Q + \overline{d} I)^ 2 + K
		\end{align*}
		for some compact operator $K$. Therefore, $C_{\psi}^ * C_{\varphi}$ is compact if and only if the operator
		\begin{align*}
			\overline{c}^ 2 T_{g^ 2} C_{\tau} Q^ 2 C_{\varphi} + 2 \overline{c d} T_{g^ 2} C_{\tau} Q C_{\varphi} + \overline{d}^ 2 T_{g^ 2} C_{\tau} C_{\varphi}
		\end{align*}
		is compact.
		
		If $C_{\tau} C_{\varphi}$ is compact, then  $T_{g^ 2} C_{\tau} C_{\varphi}$ is compact. By Lemma \ref{lemma: operator identities}, we have
		\begin{align*}
			T_{g^ 2} C_{\tau} Q C_{\varphi} = T_{g^ 2} C_{\tau} T_{\alpha} C_{\varphi} Q = T_{g^ 2} T_{\alpha \circ \tau} C_{\tau} C_{\varphi} Q,
		\end{align*}
		where $\alpha = Q \varphi$. This implies that  $T_{g^ 2} C_{\tau} Q C_{\varphi}$ is compact. By Lemma \ref{lemma: operator identities} again, we obtain that
		\begin{align*}
			T_{g^ 2} C_{\tau} Q^ 2 C_{\varphi} & = T_{g^ 2} C_{\tau} Q T_{\alpha} C_{\varphi} Q = T_{g^ 2} C_{\tau} [T_{\beta} + \alpha (0) Q] C_{\varphi} Q \\
			& = T_{g^ 2} C_{\tau} T_{\beta} C_{\varphi} Q + \alpha (0) T_{g^ 2} C_{\tau} Q C_{\varphi} Q \\
			& = T_{g^ 2} T_{\beta \circ \tau} C_{\tau} C_{\varphi} Q + \alpha (0) T_{g^ 2} T_{\alpha \circ \tau} C_{\tau} C_{\varphi} Q^ 2,
		\end{align*}
		where $\beta = Q \alpha$, which yields that $T_{g^ 2} C_{\tau} Q^ 2 C_{\varphi}$ is also compact. This completes the proof of the Lemma.
	\end{proof}

	Lemma \ref{lemma: compactness for linear fractional maps} yields the following characterization of the compactness of $C_{\varphi} C_{\psi}^ *$ and $C_{\psi}^ * C_{\varphi}$ on the Bergman space immediately.

	\begin{thm}
		Suppose $\varphi$ and $\psi$ are linear fractional maps from $\DD$ into itself. Let $\tau$ be the map associated with $\psi$ as described in Lemma \ref{lemma: representation of the adjoint}. Then:
		\begin{enumerate}[label = \textup{(\alph*)}]
			\item $C_{\varphi} C_{\psi}^ *$ is compact on the Bergman space $L_a^2$ if and only if $\varphi (\TT) \cap \psi (\TT) \cap \TT = \varnothing$, and
			\item $C_{\psi}^ * C_{\varphi}$ is compact on the Bergman space $L_a^2$ if and only if $\varphi^ {-1} (\TT) \cap \psi^ {-1} (\TT) \cap \TT = \varnothing$.
		\end{enumerate}
	\end{thm}

	\begin{proof}
		We only prove (a), since (b) can be shown similarly. In the view of Lemma \ref{lemma: compactness for linear fractional maps}, it is sufficient to show that $C_\varphi C_{\tau}$ is compact on $L_a^2$ if and only if $\varphi (\TT) \cap \psi (\TT) \cap \TT = \varnothing$.

 From the remarks in \cite[Section 11.3]{Zhu}, we have that $C_{\tau \circ \varphi}$ is compact if and only if $\tau \circ \varphi$ does have not a finite angular derivative at any point of $\TT$, which is equivalent to that the Hardy--composition operator $C_{\tau \circ \varphi}$ is compact (see \cite{CZ}). According to \cite[Theorem 2.1]{CZ}, the Hardy--composition operator $C_{\tau \circ \varphi}$ is compact if and only if $\varphi (\TT) \cap \psi (\TT) \cap \TT = \varnothing$, which gives (a).
	\end{proof}

To study when the operators $C_{\varphi} C_{\psi}^*$ and $C_{\psi}^* C_{\varphi}$ belong to the Schatten classes on the Bergman space,  we need the following lemma which is parallel to Lemma \ref{lemma: linear fractional symbols} in the Hardy space setting.
	\begin{lem}
		Suppose $\varphi$ and $\psi$ are linear fractional maps from $\DD$ into itself. Let $\tau$ be the map associated with $\psi$ as described in Lemma \ref{lemma: representation of the adjoint}. Then we have:
		\begin{enumerate}[label = \textup{(\alph*)}]
			\item  $C_{\varphi} C_{\psi}^*$ is in $S_p (L_a^ 2)$ if and only if $C_{\varphi} C_{\tau}$ is in $S_p (L_a^ 2)$, where $0 < p < \infty$;
			\item  $C_{\psi}^* C_{\varphi}$ is in $S_p (L_a^ 2)$ if and only if $C_{\tau} C_{\varphi}$ is in $S_p (L_a^ 2)$,  where $1 < p < \infty$.
		\end{enumerate}
	\end{lem}

	\begin{proof}
		The proof is the same as that of Lemma \ref{lemma: compactness for linear fractional maps} with ``Schatten $p$-class operators'' in place of ``composition operators''. However, note that the operator $R$ defined in the proof of Lemma \ref{lemma: compactness for linear fractional maps} belongs to $S_p (L_a^ 2)$ only for $1 < p < \infty$. Thus  we require $1 < p < \infty$ in (b).
	\end{proof}

With the help of the above lemma, we can establish the following necessary and sufficient conditions for  $C_{\varphi} C_{\psi}^*$ and $C_{\psi}^* C_{\varphi}$ to be in the Schatten classes on the Bergman space.

	\begin{thm}
		Suppose $\varphi$ and $\psi$ are linear fractional maps from $\DD$ into itself. Then we have:
		\begin{enumerate}[label = \textup{(\alph*)}]
			\item [\emph{(i)}] $C_{\varphi} C_{\psi}^ *$ is in $S_p (L_a^ 2)$ if and only if $\varphi (\TT) \cap \psi (\TT) \cap \TT = \varnothing$, where $0 < p < \infty$;
			\item [\emph{(ii)}] $C_{\psi}^ * C_{\varphi}$ is in $S_p (L_a^ 2)$ if and only if $\varphi^ {-1} (\TT) \cap \psi^ {-1} (\TT) \cap \TT = \varnothing$, where $1 < p < \infty$.
		\end{enumerate}
	\end{thm}

	\begin{proof}
		The proof is analogous to that of Theorem \ref{theorem: characterization of Sp for linear fractional symbols}, with the only difference being that (a) of Theorem \ref{therorem: characterization of Sp for Cphi} used in the proof is replaced by (b).
	\end{proof}

	We now turn to discussing when the operator $C_{\psi}^ * C_{\varphi}$ is in the Schatten classes, where $\varphi$ and $\psi$ are  general self-mappings of the open unit disk $\mathbb D$.
	
	\begin{thm}\label{BS}
		Let $\varphi$ and $\psi$ be two analytic self-mappings of $\DD$, and $1 \leqslant p < \infty$. Then  $C_{\psi}^ * C_{\varphi} \in S_p (L_a^ 2)$ if
		\begin{align*}
			\int_{\DD} \frac{\rd A(z)}{\left( 1 - |\varphi (z)| |\psi (z)| \right) \left( 1 - |\varphi (z)| \right)^ p \left( 1 - |\psi (z)| \right)^ p} < +\infty.
		\end{align*}
	\end{thm}

	\begin{proof}
		The proof is analogous to that of Theorem \ref{proposition: sufficient condition for Cpsi*Cphi in Sp}. In the following,  we only provide an outline of the proof. Let $P_n$ be the orthogonal projection from $L_a^2$ onto $\spa{\left\{ \sigma_0, \cdots, \sigma_n \right\}}$, where  $$\sigma_n (z) = \sqrt{n + 1} z^ n, \ \ \ n=0, 1, 2, \cdots.$$ Let $$Q_n = I - P_n,$$ where  $I$ is the identity operator on $L_a^2$ and $n$ is a positive integer. For each $f$ in $L_a^ 2$, we write
$$f(z) = \sum\limits_{j = 0}^ {\infty} {a_j \sigma_j}(z), \ \ \ \ z\in \mathbb D.$$
It follows that $$(Q_n f)(z) = \sum\limits_{j = n + 1}^ {\infty} {a_j \sigma_j}(z),$$
which yields
		\begin{align*}
			|(Q_n f) (z)| & \leqslant \left( \sum_{j = n + 1}^ {\infty} {|a_j|^ 2} \right)^ {\frac{1}{2}} \left( \sum_{j = n + 1}^ {\infty} {|\sigma_j (z)|^ 2} \right)^ {\frac{1}{2}} \\
			& \leqslant |z|^ {n + 1} \left[ \frac{1}{(1 - |z|^ 2)^ 2} + \frac{n + 1}{1 - |z|^ 2} \right]^ {\frac{1}{2}} \| f \|_{L^ 2 (\DD)} \\
			& \leqslant C |z|^ {n + 1} \left[ \frac{1}{1 - |z|^ 2} + \frac{n^ {\frac{1}{2}}}{(1 - |z|^ 2)^ {\frac{1}{2}}} \right] \| f \|_{L^ 2 (\DD)},
		\end{align*}
		where $C > 0$ is a constant.
		Now mimicking the proof of  Theorem \ref{proposition: sufficient condition for Cpsi*Cphi in Sp}, we arrive at the desired sufficient condition given in Theorem \ref{BS}.
	\end{proof}

The following result gives a criterion for $C_{\psi}^* C_{\varphi}$ to be in the Hilbert--Schmidt class on the Bergman space $L_a^2$.
	
	\begin{prop}
		Suppose that $\varphi$ and $\psi$ are analytic self-mappings of $\DD$. If
		\begin{align*}
			\int_{\DD} \int_{\DD} \frac{\rd A(z) \rd A(w)}{\big| 1 - \varphi (z) \overline{\varphi (w)} \big|^ 2 \big| 1 - \overline{\psi (z)} \psi (w) \big|^ 2} < +\infty,
		\end{align*}
then $C_{\psi}^ * C_{\varphi} \in S_2 (L_a^ 2)$.
	\end{prop}

	\begin{proof}
		Fixing an $r \in (0, 1)$ and a positive integer $N$, we have
		\begin{align*}
			 &\sum_{j, k = 0}^ N {\Big| \int_{r \DD} (C_{\varphi} \sigma_j) (z) \overline{(C_{\psi} \sigma_k) (z)} \dd A(z) \Big|^ 2}\\
& = \sum_{j, k = 0}^ N {(j + 1) (k + 1) \Big| \int_{r \DD} \varphi (z)^ j \overline{\psi (z)}^ k \dd A(z) \Big|^ 2} \\
			& = \int_{r \DD} \int_{r \DD} \sum_{j, k = 0}^ N {(j + 1) (k + 1) \varphi(z)^ j \overline{\varphi (w)}^ j \overline{\psi (z)}^ k \psi (w)^ k \dd A(z) \rd A(w)} \\
			& = \int_{r \DD} \int_{r \DD} \bigg\{ \frac{1 - \varphi (z)^ {N + 1} \overline{\varphi (w)}^ {N + 1}}{\big[ 1 - \varphi (z) \overline{\varphi (w)} \big]^2} - (N + 1) \frac{\varphi (z)^ {N + 1} \overline{\varphi (w)}^ {N + 1}}{1 - \varphi (z) \overline{\varphi (w)}} \bigg\} \\
			& \, \quad  \times \bigg\{ \frac{1 - \psi (w)^ {N + 1} \overline{\psi (z)}^ {N + 1}}{\big[ 1 - \psi (w) \overline{\psi (z)} \big]^ 2} - (N + 1) \frac{\psi (w)^ {N + 1} \overline{\psi (z)}^ {N + 1}}{1 - \psi (w) \overline{\psi (z)}} \bigg\} \dd A(z) \rd A(w).
		\end{align*}
		Letting $N \rightarrow \infty$, we obtain that
		\begin{align*}
			\sum_{j, k = 0}^ {\infty} {\Big| \int_{r \DD} (C_{\varphi} \sigma_j) (z) \overline{(C_{\psi} \sigma_k) (z)} \dd A(z) \Big|^ 2} = \int_{r \DD} \int_{r \DD} \frac{\rd A(z) \rd A(w)}{\big[ 1 - \varphi (z) \overline{\varphi (w)} \big]^ 2 \big[ 1 - \psi (w) \overline{\psi (z)} \big]^ 2}.
		\end{align*}
		It follows from the assumption and Fatou's lemma that
		\begin{align*}
			\sum_{j, k = 0}^ {\infty} {\left| \langle C_{\psi}^ * C_{\varphi} \sigma_n, \sigma_n \rangle_{L^ 2 (\DD)} \right|^ 2} & = \sum_{j, k = 0}^ {\infty} {\Big| \int_{\DD} (C_{\varphi} \sigma_j) (z) \overline{(C_{\psi} \sigma_k) (z)} \dd A(z) \Big|^ 2} \\
			& \leqslant \varliminf_{n \rightarrow \infty} {\sum_{j, k = 0}^ {\infty} {\Big| \int_{r_n \DD} (C_{\varphi} \sigma_j) (z) \overline{(C_{\psi} \sigma_k) (z)} \dd A(z) \Big|^ 2}} \\
			& \leqslant \int_{\DD} \int_{\DD} \frac{\rd A(z) \rd A(w)}{\big| 1 - \varphi (z) \overline{\varphi (w)} \big|^ 2 \big| 1 - \psi (w) \overline{\psi (z)} \big|^ 2},
		\end{align*}
		where $\{ r_n \}_{n = 1}^ {\infty}$ is any sequence of positive numbers increasing to $1$. This completes the proof.
	\end{proof}

	\begin{prop}
		Suppose that $\varphi$ and $\psi$ are analytic self-mappings of $\DD$. Then  $C_{\psi}^ * C_{\varphi} \in S_1 (L_a^ 2)$ provided that
		\begin{align*}
			\int_{\DD} \frac{\rd A(z)}{\left( 1 - |\varphi (z)|^ 2 \right) \left( 1 - |\psi (z)|^ 2 \right)} < +\infty.
		\end{align*}
	\end{prop}
	
	\begin{proof}
		As in the proof of Proposition \ref{proposition: sufficient condition for S1, Hardy space}, the desired result follows from an estimate of
		\begin{gather*}
			\sum_{n = 1}^ {\infty} {\Big| \left\langle C_{\psi}^ * C_{\varphi} e_n, e_n \right\rangle_{L^ 2 (\DD)} \Big|},
		\end{gather*}
		where $\{ e_n \}_{n = 1}^ {\infty}$ is any orthonormal basis of $L_a^ 2$.
	\end{proof}
	
The final result concerning the operator $C_{\psi}^*C_\varphi$ on the Bergman space is the following proposition, which provides a necessary condition for this operator to belong to the trace class.

	\begin{prop}
		Suppose $\varphi$ and $\psi$ are analytic self-mappings of $\DD$. If $C_{\psi}^ * C_{\varphi}$ belongs to $S_1 (L_a^ 2)$, then the limit
		\begin{align*}
			\lim_{r \rightarrow 1^ -} {\int_{\DD} \frac{\rd A(z)}{\big[ 1 - r \varphi (z) \overline{\psi (z)} \big]^ 2}}
		\end{align*}
		exists. In this case, the above quantity is the trace of $C_{\psi}^ * C_{\varphi}$.
	\end{prop}
	
	\begin{proof}
		Computing the series
		\begin{align*}
			\sum_{n = 0}^ {\infty} {r^ n \left\langle C_{\psi}^ * C_{\varphi} \sigma_n, \sigma_n \right\rangle_{L^ 2 (\DD)}}, \ \ \ \  0 \leqslant r < 1,
		\end{align*}
		and then letting $r \rightarrow 1^ -$, we can obtain the desired trace formula.
	\end{proof}
	
	In the remainder of this section, we consider the operator $C_{\varphi} C_{\psi}^*$ on the Bergman space $L_a^2$. The next theorem gives a  criterion for
	$C_{\varphi} C_{\psi}^*$ to be in $S_p (L_a^2)$ with $p\geqslant 2$.

	\begin{thm}
		Suppose that $2 \leqslant p < \infty$,  $\varphi$ and $\psi$ are analytic self-mappings of $\DD$. If
		\begin{align}
			\int_{\DD} \int_{\DD} \frac{\rd A(w) \rd A(z)}{\big| 1 - \overline{\varphi (z)} \psi (w) \big|^ {2 p}} < +\infty,\label{ineq5}
		\end{align}
		then $C_{\varphi} C_{\psi}^ * \in S_p (L_a^2)$.
		In particular, Condition \textup{(\ref{ineq5})} is also necessary  when $p = 2$.
	\end{thm}
	
	\begin{proof}
		Applying Lemma \ref{lemma: Cphi* Cphi is a Toeplitz opeartor, Bergman space}, we get that
		\begin{align*}
			\left( C_{\varphi} C_{\psi}^ * \right)^ * \left( C_{\varphi} C_{\psi}^ * \right) = C_{\psi} T_{\mu_{\varphi}} C_{\psi}^ *,
		\end{align*}
		where $$\rd \mu_{\varphi} = \rd A \circ \varphi^ {-1}.$$ Then using an approach similar to that in Theorem  \ref{proposition: sufficient condition for Cphi Cpsi* in Sp}, we can estimate
		\begin{align*}
			\sum_{n = 1}^ {\infty} {\left| \left\langle C_{\psi} T_{\mu_{\varphi}} C_{\psi}^ * e_n, e_n \right\rangle_{L^ 2 (\DD)} \right|^ {\frac{p}{2}}},
		\end{align*}
		 where $\{ e_n \}_{n = 1}^ {\infty}$ is any orthonormal basis of $L_a^ 2$. The remaining steps are omitted for brevity.
	\end{proof}
	
	Recall that Lemma \ref{lemma: trace formula} plays a crucial role in the proof of Proposition \ref{proposition: trace formula, II}. In order to establish a corresponding trace formula for the  operator $C_{\varphi} C_{\psi}^ *$ on the Bergman space $L_a^ 2$,  an analogue of Lemma \ref{lemma: trace formula} is required. The \textit{Berezin transform} of a bounded linear  operator $T$ on the Bergman space $L_a^ 2$, denoted by $\mathcal{B} T$ (to  keep the notation simple, we reuse the notation from the Hardy space setting), is defined by
	\begin{align*}
		\mathcal{B} T (z) = \langle T k_z, k_z \rangle_{L^ 2 (\DD)}, \qquad z \in \DD,
	\end{align*}
	where $$k_z (w) = \frac{1 - |z|^ 2}{\left( 1 - \overline{z} w \right)^ 2}$$ is the normalized reproducing kernel for $L_a^ 2$ at $z\in \mathbb D$.
	
	\begin{lem}[{\cite[Corollary 6.5]{Zhu}}]\label{LF}
		Let $T$ be a bounded linear operator on $L_a^ 2$. If $T$ belongs to $S_1 (L_a^ 2)$, then the integral
		\begin{align*}
			\int_{\DD} \mathcal{B} T (z) \, \frac{\dd A(z)}{(1 - |z|^ 2)^ 2}
		\end{align*}
		converges, which is exactly the trace of $T$.
	\end{lem}
	
	Using this lemma and preceding as in the proof of Proposition \ref{proposition: trace formula, II}, we obtain the following necessary condition for the operator  $C_{\varphi} C_{\psi}^*$  to belong to the trace class.
	
	\begin{prop}
		Suppose that $\varphi$ and $\psi$ are analytic self-mappings of $\DD$. If $C_{\varphi} C_{\psi}^ *$ belongs to $S_1 (L_a^2)$, then the integral
		\begin{align*}
			\int_{\DD} \frac{\rd A(z)}{\big[ 1 - \varphi (z) \overline{\psi (z)} \big]^ 2}
		\end{align*}
		converges. In this case, the above quantity is the trace of $C_{\varphi} C_{\psi}^ *$.
	\end{prop}
	
	\begin{proof}
		Note that the Berezin transform of $C_{\varphi} C_{\psi}^* $ is given by
		\begin{align*}
			\mathcal{B} (C_{\varphi} C_{\psi}^ *) (z) = \left\langle C_{\varphi} C_{\psi}^ * k_z, k_z \right\rangle_{L^ 2 (\DD)} = (1 - |z|^ 2)^ 2 \left\langle C_{\psi}^ * K_z, C_{\varphi}^ * K_z \right\rangle_{L^ 2 (\DD)}, \ \ \ \ z \in \DD.
		\end{align*}
By a routine calculation similar to that in Proposition \ref{proposition: trace formula, II},  we obtain  $$C_{\psi}^ * K_z = K_{\psi (z)} \ \ \ \  \text{and} \ \ \ \ C_{\varphi}^* K_z = K_{\varphi (z)}.$$ Consequently, we have
		\begin{align*}
			\mathcal{B} (C_{\varphi} C_{\psi}^ *) (z) = \frac{(1 - |z|^ 2)^ 2}{\big[1 - \varphi (z) \overline{\psi (z)} \big]^ 2}, \ \ \ \ z\in \mathbb D.
		\end{align*}
		Then the desired result follows from Lemma \ref{LF} immediately.
	\end{proof}
\vspace{3mm}
\subsection*{Acknowledgment}
This work was supported by the National Natural Science Foundation of China
(12371125)  and the Chongqing Natural Science Foundation (CSTB2024NSCQ-MSX0177).

\end{document}